\documentclass[12pt,a4paper]{amsart}
\calclayout
\usepackage{float}
\usepackage{mathrsfs}
\usepackage{comment}
\usepackage[unicode]{hyperref}

\usepackage{amsmath,amsxtra,amssymb,amsthm,amsfonts,mathtools}
\usepackage{mathrsfs}
\usepackage[T1]{fontenc}
\usepackage[utf8]{inputenc}
\usepackage{mathtools}   
\usepackage{multirow}
\usepackage{tikz,quantikz}
\usepackage{physics}
\usepackage{bbm}
\usepackage{subfigure}
\usepackage{comment}
\usepackage{appendix}
\usetikzlibrary{arrows, automata}
\usepackage{color,hyperref}
\numberwithin{equation}{section}
\allowdisplaybreaks
\theoremstyle{plain}
\usepackage{amssymb}
\usepackage{mathtools}
\allowdisplaybreaks

\newcommand{\al}{\alpha}
\newcommand{\be}{\beta}
\newcommand{\ga}{\gamma}

\newcommand{\de}{\delta}

\newcommand{\Si}{\Sigma}

\newcommand{\mc}{\mathcal}
\newcommand{\mb}{\mathbb}

\newcommand{\wt}{\widetilde}

\newcommand{\ZR}{\mathbb{R}}

\newcommand{\ppv}{\text{p.v.}}

\newcommand{\R}{\mathbb{R}}

\newcommand{\dist}{{\rm dist}}

\newtheorem{theorem}{Theorem}[section]
\newtheorem{definition}[theorem]{Definition}

\newtheorem{remark}[theorem]{Remark}
\newtheorem{lemma}[theorem]{Lemma}
\newtheorem{prop}[theorem]{Proposition}

\begin{document}
\title[Regional fractional Laplacian with multi-singular critical perturbation]{Heat kernel estimates for regional fractional Laplacians with multi-singular critical potentials
in $C^{1, \beta}$ open sets
}
\author[R. Song, P. Wu and S. Wu]{Renming Song, Peixue Wu and Shukun Wu}
\address{Renming Song\\ Department of Mathematics\\ University of Illinois Urbana-Champaign \\ Urbana \\ IL 61801\\ USA}
\email{rsong@illinois.edu}
\thanks{The research of Renming Song is supported in part by a grant from the Simons Foundation (\#429343, Renming Song)}

\address{Peixue Wu\\ Institute for Quantum computing, Department of Applied Mathematics\\ University of Waterloo\\ON, Canada}
\email{p33wu@uwaterloo.ca}
\thanks{The research of Peixue Wu is supported by the Canada First Research
Excellence Fund (CFREF)}
\address{Shukun Wu\\ Department of Mathematics \\ Indiana University\\ Bloomington \\ IN \\ USA }
\email{shukwu@iu.edu}

\begin{abstract}
Let $D$ be an open set of $\R^d$, $\alpha\in (0, 2)$ and let $\mathcal{L}_{\alpha}^D$ be the generator of the censored $\alpha$-stable process in $D$. In this paper, we establish sharp two-sided heat kernel estimates for $\mathcal{L}_{\alpha}^D-\kappa$,  with $\kappa$ being a non-negative critical potential and $D$ being a $C^{1, \beta}$ open set, $\beta \in ((\alpha-1)_+,1]$. The potential $\kappa$ can exhibit multi-singularities and our regularity assumption on $D$ is weaker than the regularity assumed in earlier literature on heat kernel estimates of fractional Laplacians.  
\end{abstract}

\maketitle
\section{Introduction}\label{s:intro}
The transition density of a Markov process $\{X_t\}_{t\ge0}$, when it exists, encapsulates all the statistical information about the process. For a few notable examples, such as Brownian motion, the Cauchy process, and the Ornstein-Uhlenbeck process, explicit formulas for the transition density are known. However, for most Markov processes, such explicit expressions are not available. The transition density, often referred to as the heat kernel of the process's generator, plays a central role in both analysis and probability theory. Consequently, deriving sharp two-sided estimates for heat kernels has become a significant focus of research in these fields.

Assume $\alpha\in (0, 2)$. An isotropic $\alpha$-stable process $\{Z_t\}_{t\ge0}$ on $\R^d$ is a  L\'evy process with
\begin{equation*}
\mb{E}\exp(i\xi \cdot (Z_t-Z_0)) = \exp(-t|\xi|^{\alpha}), 
\quad \xi \in \mb{R}^d, t\ge 0.
\end{equation*}
The generator of $Z$ is the  fractional Laplacian
$-(-\Delta)^{\alpha/2}$ which can be written as
$$
-(-\Delta)^{\alpha/2}f(x)=\mc{A}(d,-\alpha)\text{ p.v.}\int_{\R^d}\frac{f(y)-f(x)}{|y-x|^{d+\alpha}}dy,
$$
where
\begin{align}\label{e:constantA}
\mc{A}(d,-\alpha)=\frac{\alpha2^{\alpha-1}\Gamma((\alpha+d)/2)}{\pi^{d/2}\Gamma(1-\alpha/2)}.
\end{align}
It is well known (see \cite{BG60} for instance) that the heat kernel $p(t,x,y)$ of 
$Z$ has the following estimates:
there exists $\delta \in (0,1]$ such that
\begin{equation*}
\delta \big( t^{-d/\alpha} \wedge \frac{t}{|x-y|^{d+\alpha}} \big) \le p(t,x,y) \le \delta^{-1} \big( t^{-d/\alpha} \wedge \frac{t}{|x-y|^{d+\alpha}} \big),
\end{equation*}
for all $(t,x,y) \in (0,\infty) \times \mb{R}^d \times \mb{R}^d$. 
In this paper, we will use
$A\asymp B$ to denote that $A/B$ is bounded between two positive constants. So the estimates above can be written as $p(t,x,y)\asymp t^{-d/\alpha} \wedge \frac{t}{|x-y|^{d+\alpha}}$. 

Suppose $D$ is an open subset of $\mb{R}^d$ and $Z^D$ is the killed $\alpha$-stable process in $D$. Using the strong Markov property one can easily show that $Z^D$ has a heat kernel $p^{D}(t,x,y)$. Due to complication near the boundary, sharp two-sided estimates of $p^{D}(t,x,y)$ were established only recently in \cite{CKS10}. The main result of \cite{CKS10} is as follows. 
If $D\subset \mb{R}^d$ is a $C^{1,1}$ open set, then $p^D$ admits the following estimates:
\begin{itemize}
\item 
For any $T>0$, on $(0,T)\times D \times D$, 
\begin{equation}\label{e:3}
p^{D}(t,x,y) \asymp \left(1 \wedge \frac{\delta_{D}(x)}{t^{1/\alpha}}\right)^{\frac{\alpha}{2}} \left(1 \wedge \frac{\delta_{D}(y)}{t^{1/\alpha}}\right)^{\frac{\alpha}{2}} p(t,x,y),
\end{equation}
where $\delta_D(x)$ denotes the distance between $x$ and $D^c$.
\item 
Suppose in addition that  $D$ is bounded. Then for any $T>0$, on $[T,\infty) \times D \times D$, 
\begin{equation}\label{e:4}
p^{D}(t,x,y) \asymp e^{-\lambda_1 t} \delta_D(x)^{\frac{\alpha}{2}}\delta_D(y)^{\frac{\alpha}{2}},
\end{equation} 
where $-\lambda_1$ is the largest eigenvalue of 
the generator of $Z^D$.
\end{itemize}
In the case when $D$ is only assumed to a $\kappa$-fat open set, 
the following sharp two-sided small time estimates for 
$p^D$ were obtained 
in  \cite{Bogdan10}: for any $T>0$, on $(0,T)\times D \times D$,
\begin{equation}\label{factorization type}
p^D(t,x,y)\asymp \mb P^x(\tau_D>t)\mb P^y(\tau_D>t)p(t,x,y),
\end{equation}
where $\tau_D$ is the first exit time of $Z$ from $D$. This kind of estimates can be regarded as an approximate factorization of $p^D(t,x,y)$. It is also known as a Varopoulos-type 
estimate,  since Varopoulos first observed this phenomenon for killed Brownian motion in \cite{Varopoulos03}.
This kind of approximate factorization has been generalized to more general symmetric Markov processes, see, for instance, \cite{CKS14}. 
In \cite{CKSV20}, this kind of approximate factorization has been generalized to a large class of 
 jump processes with critical killing.  We mention in passing that getting explicit sharp two-sided estimates from the 
 approximate factorization
 is no easy task, and that it is impossible to establish explicit sharp two-sided estimates for the Dirichlet heat kernel $p^D(t, x, y)$ for a general $\kappa$-fat open set $D$. 
 
 In the case when $D$ is a $d$-set, sharp two-sided heat kernel estimates for the reflected $\alpha$-stable process in $\overline{D}$ were  obtained in \cite{CK03}. In \cite{CKS10b}, sharp two-sided heat kernel estimates for the censored
 $\alpha$-stable process in $C^{1, 1}$ open sets were established.

The main purpose of this paper is to establish sharp explicit 
two-sided heat kernel estimates 
for a large class of processes which are 
obtained from Feynman-Kac transforms of censored $\alpha$-stable processes with multi-singular critical killing potentials
in $C^{1, \beta}$ open set $D$, $\beta \in ((\alpha-1)_+,1]$.  

The setup of this paper is as follows. Assume $d\ge 2$ and $\alpha\in (0, 2)$. Suppose that $D$ is a $C^{1, \beta}$
open set in $\R^D$ with $\beta \in ((\alpha-1)_+,1]$.  See Section \ref{sec:preliminary} for the definition of $C^{1,\beta}$ open set. The boundary of $\partial D$ may not be connected. We assume
that 
\begin{align}\label{assumption:boundary}
    \partial D = \bigcup_{j=1}^{m_0} \Si_j,
\end{align}
where $m_0\ge 1$,  and $\Si_j$, $1\le j\le m_0$, are the connected components of $\partial D$. 
Let $X$ be the reflected $\alpha$-stable process in $\overline{D}$. Let $\tau_D=\inf\{t>0: X_t\notin D\}$ be the
first exit time of $D$ for $X$.  The killed process $X^D$ is defined by $X^D_t=X_t$ if $t<\tau_D$ and $X^D_t=\partial$
if $t\ge \tau_D$, where $\partial$ is a cemetery point. $X^D$ is the censored $\alpha$-stable process defined in
\cite{BBC03}.  Let $\mathcal{L}_{\alpha}^D$ be the generator of $X^D$. 

For $x\in \R^d$, we use $\delta_{\Si_j}(x)$ to denote the distance between $x$ and $\Si_j$. We assume that
\begin{align}\label{def:killing function}
    \kappa(x) = \sum_{j=1}^{m_0} \lambda_{j} \delta_{\Sigma_{j}}(x)^{-\alpha} + \sum_{l=1}^{m_1} \mu_l |x-x_l|^{-\alpha},
    \quad x\in D,
\end{align}
where (i) 
$m_1$ is a non-negative integer, $\{x_l\}$ are distinct points in $D$, 
 and for $1\le l\le m_1$, $\mu_l$ is a positive constant; (ii) for $1\le j\le m_0$,
$\lambda_j$ is a non-negative number which is further assumed to be positive in the case $\alpha\in (0, 1]$. 
For any non-negative Borel function $f$ on $D$, we define
\begin{align}\label{e:semig}
T_t^{D, \kappa} f(x):=\mb E^x\left(e^{-\int^{t\wedge \tau_D}_0\kappa(X_s)ds}f(X^D_t)\right), \quad t\ge 0, x\in D.
\end{align}
$(T^{D, \kappa}_t)_{t\ge 0}$ is a sub-Markov semigroup and it is called the Feynman-Kac semigroup of $X^D$ with killing potential $\kappa$. The generator of this semigroup is $\mathcal{L}_{\alpha}^D-\kappa$, where $\kappa$ is understood as a multiplication operator. Let $Y$ be the Hunt process associated with 
$(T^{D, \kappa}_t)_{t\ge 0}$ and let $\zeta$ be the lifetime of $Y$. The state space of 
$Y$ is $D_{\kappa}=D\setminus\cup^{m_1}_{l=1}\{x_l\}$. One can easily check that $Y$ has a transition density $q^{D, \kappa}(t, x, y)$. 
Our main result is as follows:

\begin{theorem}\label{thm:main}
The heat kernel $q^{D, \kappa}(t,x,y)$ admits the following sharp two-sided estimates: for any $T>0$, 
\begin{equation}
q^{D, \kappa}(t,x,y) \asymp B_{\kappa}(x,y) \left(t^{-\frac{d}{\alpha}} \wedge \frac{t}{|x-y|^{d+\alpha}}\right) \quad \mbox{on } (0,T) \times D_\kappa \times D_\kappa,
\end{equation}
where 
$$
     B_{\kappa}(x,y) = \prod_{j=1}^{m_0} \left(1\wedge \frac{\delta_{\Si_j}(x)}{t^{1/\alpha}}\right)^{p_{j}}\left(1\wedge \frac{\delta_{\Si_j}(y)}{t^{1/\alpha}}\right)^{p_{j}} \prod_{l=1}^{m_1}  \left(1\wedge \frac{|x - x_l|}{t^{1/\alpha}}\right)^{q_{l}}\left(1\wedge \frac{|y-x_l|}{t^{1/\alpha}}\right)^{q_{l}},
$$
with $p_{j}$ determined by $C_0(d,\alpha,p_{j}) = \lambda_{j}$ and $q_l$ determined by $C_1(d,\alpha,q_{l}) = \mu_{l}$, 
where $C_0,C_1$ are defined in \eqref{constant 1: C_0(d,p)} and \eqref{constant 2: C_1(d,p)}.
\end{theorem}

\begin{remark}{\rm Theorem \ref{thm:main} generalizes the main results of \cite{CKSV20, JW20} in the sense that we allow
multi-singular critical potentials.
}
\end{remark}

\begin{remark}{\rm 
In the case $\alpha\in (1, 2)$ and $\kappa\equiv 0$, Theorem \ref{thm:main} generalizes the heat kernel estimates for
censored $\alpha$-stable processes to $C^{1, \beta}$ open set $D$ with  $\beta \in ((\alpha-1)_+,1]$. Even in this special case, our result is new. 
}
\end{remark}

\begin{remark}
{\rm 
As in \cite{CKSV20} (see the definition of $\mathcal{H}_\alpha$ on \cite[p. 234]{CKSV20} and the definition of $\mathcal{G}_\alpha$ on \cite[p. 250]{CKSV20}), we could allow a lower order perturbation in each of the terms in the definition of $\kappa$. 
We do not consider this more general case since the extension is pretty routine.
Had we allowed this routine extension, as a consequence of Theorem \ref{thm:main} in the case $m_1=0$ and $\lambda_1=\cdots=\lambda_{m_0}$, we could
get sharp explicit two-sided heat kernel estimates for the killed $\alpha$-stable processes in  $C^{1, \beta}$ open set $D$ with  $\beta \in ((\alpha-1)_+,1]$.  See \cite[Remark 3.4]{CKSV20}.  In \cite{Kim14, Kim19}, the sharp explicit two-sided heat estimates for $Z^D$ were extended from
$C^{1, 1}$ open set to $C^{1, \beta}$ open set with $\beta\in (\alpha/2, 1]$.  The condition $\beta \in ((\alpha-1)_+,1]$
is weaker than $\beta\in (\alpha/2, 1]$.
}
\end{remark}

\begin{remark}
{\rm Theorem \ref{thm:main} can be regarded as the non-local counterpart of the main result of \cite{Filipas09}. In
\cite{Filipas09}, the boundary of the open set $D$ is assumed to be the finite union of finitely many distinct boundary-less 
$C^2$ hypersurfaces of co-dimension between 1 and $d-1$. In the present paper, every component of $\partial D$ is of co-dimension 1. We have not been able to extend our argument to deal with the case that some components of $\partial D$
are of co-dimension between 2 and $d-1$.  
}
\end{remark}

\begin{remark}{\rm 
Recently, \cite{Antonio24, Florian24, KM24} confirmed that weaker boundary regularity is possible for various problems related to nonlocal operators. However, they have not considered heat kernel estimates yet.
}
\end{remark} 
The two main contributions of this paper are as follows. (i) Compared with \cite[Theorem 3.2]{CKSV20}, we weakened  the $C^{1, 1}$ regularity on $D$ to $C^{1, \beta}$ with $\beta \in ((\alpha-1)_+,1]$. This is a natural assumption since the generator of our process is of order $\alpha$. (ii) Compared with \cite{CKSV20, JW20}, we allow multi-singular critical killing potentials. 

The technique of this paper is a mixture of probabilistic and analytic arguments. We use a perturbation argument to construct some appropriate barrier functions and then we use Dynkin's formula to get sharp explicit two-sided estimates for the survival probability of $Y$.  Once this is done, we can use approximate factorization result of 
of \cite[Theorem 2.23]{CKSV20} to get the conclusion of Theorem \ref{thm:main}.

The rest of this paper is organized as follows. In Section 2, we give some preliminaries, including the Varopoulos-type heat kernel estimates developed in \cite{CKSV20}. We also recall the harmonic profiles for the 
operator $\mathcal{L}_{\alpha}^D-\kappa$ in the case $D$ is the upper half-space and $\kappa(x)=\lambda x_d^{-\alpha}$,  and the case $D=\R^d$ and $\kappa(x)=\mu|x|^{-\alpha}$. In Section 3, we present a perturbation argument for $C^{1,\beta}$  open sets, $\beta \in ((\alpha-1)_+,1]$, the key estimates are given by Lemmas \ref{key lemma: distance comparison} and \ref{key prop: perturbation}. In Section 4, we give the proof of Theorem \ref{thm:main}. Two technical results are proved in the Appendix.

In this paper, we use the following notational conventions. Capital letters such as $ A_i, C_i, R_i$ denote  constants whose values remain fixed throughout the paper. When we specify the dependence of the constants, we omit the dependence on $\alpha$ and $d$. 
Lowercase letters such as $ c_i, r_i$ represent constants whose value may 
vary from the statement and proof of one result to another. For two functions $f$ and $g$,  
        $f \asymp g $ means that there exist constants $c_1 > c_2 > 0$ such that
        $c_2 g \leq f \leq c_1 g$;
        $ f \lesssim g $ means that there exists a constant $ c > 0$ such that
        $f \leq c g$.
        
\section{Preliminaries}\label{sec:preliminary}

Throughout this paper, we assume $d\ge 2$ and $\alpha\in (0, 2)$. For any $r>0$ and $r_1>r_2>0$, $x\in D$, we define
$B(x,r)=\{z\in \R^d: |z-x|<r\}$ and  $\mc A(x,r_1,r_2):= \{z\in \R^d: r_2<|z-x|<r_1\}$.
We first recall the definition of $C^{1, \beta}$ open sets. 

\begin{definition}\label{d:c1beta}
Let $\beta\in (0, 1]$. An open set $D\subset\R^d$ is said to be a $C^{1,\beta}$ open set if there exist a localization radius 
$R_0>0$  and a constant
$\Lambda > 0$ such that for every $z \in \partial D$, there exist a $C^{1,\beta}$ function $\psi (\cdot)=\psi_z(\cdot): \mathbb{R}^{d-1} \to \mathbb{R}$ satisfying $\psi(\wt z)= 0, \nabla \psi(\wt z) = 0$, $||\nabla \psi||_{\infty} \le \Lambda$,  
\begin{equation} \label{C1 beta function}
||\psi||_{1,\beta}:= \sup_{\widetilde{x} \neq \widetilde{y} \in \mathbb{R}^{d-1}} \frac{|\nabla \psi(\widetilde{x})-\nabla \psi(\widetilde{y})|}{|\widetilde{x}- \widetilde{y}|^{\beta}} \le \Lambda 
\end{equation}
and an orthonormal coordinate system $CS_z$ with its origin at $z$ such that
\begin{equation} \label{local coordinate system}
B(z,R_0) \cap D = B(z,R_0) \cap \{y= (\widetilde{y},y_d) \in CS_z: y_d > \psi(\widetilde{y})\}.
\end{equation}
\end{definition}

$(R_0, \Lambda)$ is called the \textit{characteristics} of the $C^{1, \beta}$ open set $D$. Note that $\mb R^d$ is trivially a $C^{1, \beta}$ open set for any $\beta \in (0, 1]$. 
If in Definition \ref{d:c1beta}, we only assume $\psi$ is Lipschitz, then we say $D$ is a Lipschitz open set.

From now on, unless explicitly stated otherwise, we assume $D$ is a $C^{1,\beta}$ open set with $\beta\in ((\alpha-1)+, 1]$. 
As in Section \ref{s:intro}, we assume the boundary $\partial D$ is given by \eqref{assumption:boundary}. Let $X$ be the reflected $\alpha$-stable process in $\overline{D}$. It follows from \cite{CK03} that $X$ has a transition density $p(t, x, y)$
which admits the following estimates: For any $T>0$, on $(0, T]\times \overline{D}\times  \overline{D}$,
\begin{align}\label{e:stabledensity}
p(t,x,y) \asymp t^{-\frac{d}{\alpha}} \wedge 
\frac{t}{|x-y|^{d+\alpha}}.
\end{align}
Let $\mathcal{L}_{\alpha}^{\overline{D}}$ be the generator of $X$. 
For any function $f$ on $\overline{D}$ and $x\in \overline{D}$, 
we define
\begin{align}\label{e:pvgenerator}
L_\alpha^{D} f(x):=\mc{A}(d,-\alpha)\text{ p.v.}\int_{D}\frac{f(y)-f(x)}{|y-x|^{d+\alpha}}dy,
\end{align}
wherever the right hand side exists, where $\mc{A}(d,-\alpha)$ is the constant defined in \eqref{e:constantA} and
$$
\text{ p.v.}\int_{D}\frac{f(y)-f(x)}{|y-x|^{d+\alpha}}dy:=\lim_{\epsilon\downarrow 0}\int_{D\cap B(x, \epsilon)}\frac{f(y)-f(x)}{|y-x|^{d+\alpha}}dy.
$$
Let $C^2(\overline{D})$ denote the family of functions that are restrictions to $\overline{D}$ of $C^2$ functions on $\R^d$. 
It follows from \cite[Theorems 5.3 and 6.1]{Guan-Ma06} that $C^2(\overline{D})$ is contained in the domain of definition of
both $\mc{L}_\alpha^{\overline{D}}$ and $L_\alpha^{D}$, and these two operators coincide on $C^2(\overline{D})$.
Using these facts and \eqref{e:stabledensity}, one can easily check that $X$ satisfies the assumptions of \cite[Subsection 2.1]{CKSV20} and the following condition which plays a crucial role in the establish explicit estimates for the survival probability of the process $Y$ to be introduced below:

{\bf (U)}: There exists $r_0 \in (0, \infty]$ such that for any $\frac{1}{2}<b<a<1$ and $r \le  r_0$, there exists $c(a,b,r_0)>0$ such that
\begin{equation}\label{key upper bound}
\inf_{f \in \mc D\big(\overline{\mc A(x,ar, br)}\cap \overline{D}, \mc A(x,r,\frac{1}{2}r)\cap \overline{D} \big)} \sup_{z \in \overline{D}} \mc L_\alpha^{\overline{D}} f(z) \le \frac{c(a,b,r_0)}{r^\alpha},
\end{equation}
where,  for compact subset $K$ of $\overline{D}$ and open subset $U$ of $\overline{D}$ satisfying $K \subseteq U$,
$\mc D(K,U)$ is the class of $C^2$ functions $f$ with $0\le f \le 1$, $f = 1$ on $K$ and $f = 0$ on $U^c$.

Let $\kappa$ be the function defined in \eqref{def:killing function}. It is easy to check that $\kappa(x)dx$ belongs to
the class $\mathbf{K}_1(D)$ (defined in \cite[Definition 2.12]{CKSV20})  associated with $X$. Let $(T^{D, \kappa}_t)_{t\ge 0}$
be the semigroup defined in \eqref{e:semig} and let $Y$ be the Hunt process associated with $(T^{D, \kappa}_t)_{t\ge 0}$.
The state space of $Y$ is $D_{\kappa}:=D\setminus \cup^{m_1}_{l=1}\{x_l\}$. We use $\zeta$ to denote the lifetime of $Y$. 
By \cite[Proposition 2.14 and the argument before it]{CKSV20}, $Y$ 
has a transition density $q^{D, \kappa}(t,x,y)$ with respect to the Lebesgue measure.
Furthermore,  for fixed $x\in D_\kappa$, 
$(t,y)\mapsto q^{D, \kappa}(t,x,y)$ is continuous; 
and for fixed $y\in D_\kappa$, 
$(t,x)\mapsto q^{D, \kappa}(t,x,y)$ is continuous. The following estimates of $q^{D, \kappa}(t, x, y)$ are proved in \cite[Theorem 2.22]{CKSV20}:

\begin{theorem}\label{t:vthke}
(Varopoulos-type heat kernel estimates)
For any $T>0$,  
on $(0, T]\times D_\kappa\times D_\kappa$, we have 
\begin{align*}
q^{D, \kappa}(t,x,y) \asymp \mb{P}^x(\zeta > t)\mb{P}^y(\zeta > t)p(t,x,y).
\end{align*}
\end{theorem}

For any Borel set $A\subset D$, define
\begin{equation*}
\tau_A:= \inf \{t>0:Y_t \notin A\}.
\end{equation*}
The following result, which is a consequence of \cite[Corollary 2.16]{CKSV20},   will be useful in proving Lemma \ref{lifetime comparison} below.
\begin{lemma}\label{lemma:interior}
Let $T>0$ and $\alpha\in (0,1)$.  
For $(t, x)\in (0, T)\times D_\kappa$ satisfying 
$B(x,t^{1/\alpha}) \subseteq D_\kappa$, it holds that
$$
\inf_{z \in B(x,at^{1/\alpha})}\mb{P}^z(\tau_{B(x,t^{1/\alpha})} > t) \asymp 1.
$$ 
\end{lemma}

For $r>0$, we define
\begin{align}
\mc U(x,r):= B(x,r)\cap D_\kappa = \{z \in D_\kappa: |z-x|<r\}. \label{def:ball}
\end{align} 
Using condition {\bf (U)} and Lemma \ref{lemma:interior}, we can prove the following estimates,
which provide a pathway to proving explicit estimates of $\mb{P}^x(\zeta > t)$.
\begin{lemma}\label{lifetime comparison} 
Let $T \in (0,\infty)$ and $a\in (0,1]$. 
For $(t, x)\in (0, T)\times D_\kappa$
with $\delta_{D_\kappa}(x) < at^{1/\alpha}$, it holds that
$$
\mb{P}^x(\zeta > t) 
\asymp \mb{P}^x(Y_{\tau_{\mc U(x,at^{1/\alpha})}} \in D_\kappa)
\asymp \mb P^x(Y_{\tau_{\mc U(x,at^{1/\alpha})}} \in W),$$ 
where $W$ is a ball satisfying
\begin{itemize}
\item $W \subseteq D_\kappa$. 
\item radius of $W$ is comparable to $t^{1/\alpha}$.
\item $\dist(W, \mc U(x,at^{1/\alpha}) )\asymp t^{1/\alpha}$.
\item $\dist(W, D^c_\kappa) \gtrsim t^{1/\alpha}$.
\end{itemize} 
\end{lemma}
The proof is essentially contained in \cite[Lemmas 2.20 and 2.21]{CKSV20}. We provide a simplified proof in Appendix \ref{app:proof of lifetime}. This lemma allows us to reduce the lifetime estimate to exit distribution estimates, which can be estimated using Dynkin's formula. See Figure \ref{figure0} for an illustration of $W$.

\begin{figure}
\begin{tikzpicture}
\draw[thick] plot[smooth cycle] coordinates {(-4,3) (3,5) (4,-2) (0,-5) (-5,-3)};

\draw[thick] (2,0) circle (1cm);
\node at (2,0) {$\cdot$}; 
\draw[->, thick] (3,0) -- (4.2,0); 
\draw[->, thick] (0.5,-3.5) -- (2,-1);
\node at (3.6,0.5) {$\asymp t^{1/\alpha}$};
\node at (2,-2.65) {$\asymp t^{1/\alpha}$};
\node at (-0.2,-4.65) {$x \cdot$};
\draw[thick] plot[smooth] coordinates {
    (-1.5,-4.85) 
    (-0.75, -3.8)
    (0,-3.5)
    (0.75, -3.7)
    (1.5,-4.45)
};
\node[below right] at (-3, 0) {$D$};
\node[below right] at (-1.5,-2.8) {$\mc U(x,at^{1/\alpha})$};
\node[below right] at (1.9, 0.5) {$W$};
\end{tikzpicture}
\caption{Illustration of the choice of $W$.}
\label{figure0}
\end{figure}

In the remainder of this section, we recall some harmonic functions for $L^D_\alpha-\kappa$ in two special cases:
$D =\mb R_+^d = \{y= (\widetilde{y},y_d) \in \mathbb{R}^{d-1} \times \mathbb{R}: y_d > 0\}$ and $\kappa(x)=\lambda x_d^{-\alpha}$, and the case $D=\mb R^d$ and $\kappa(x)=\mu |x|^{-\alpha}$. 
In both cases, the harmonic functions are given 
by $\delta_{D_\kappa}^p(x)$, power of the distance to the boundary of $D_\kappa$.
We will call these harmonic functions \textit{harmonic profiles}.

The following result  is due to \cite[(5.4)]{BBC03}: 
\begin{prop} \label{harmonic: 1 codimension}
Suppose $D= \mathbb{R}_+^d$. For any $p \in (0,\alpha)$, 
\begin{equation} 
L^D_\alpha x_d^p -C_0(d,\alpha,p)x_d^{p-\alpha} = 0, \quad x\in D,
\end{equation}
where 
\begin{equation}\label{constant 1: C_0(d,p)}
C_0(d,\alpha,p):= \mathcal{A}(d,-\alpha) \frac{\omega_{d-1}}{2}\beta(\frac{\alpha+1}{2},\frac{d-1}{2}) \gamma(\alpha,p), 
\end{equation}
and $\beta(\cdot,\cdot)$ is the Beta function, $\omega_{d-1}$ is the area of the unit sphere in $\mathbb{R}^{d-1}$ and $\gamma(\alpha,p) = \int_{0}^1 \frac{(t^p-1)(1-t^{\alpha-p-1})}{(1-t)^{1+\alpha}}dt$. 
\end{prop}
The following result  is due to \cite[Theorem 1]{Ferrari12}:
\begin{prop} \label{harmonic: higher codimension}
Suppose $D=\mb R^d$. For any $p\in (0,\alpha)$, we have
\begin{equation} 
L^D_\alpha |x|^p- C_1(d,\alpha,p)|x|^{p-\alpha} = 0, \quad x\in \mb R^d\setminus \{0\},
\end{equation}
where 
\begin{equation}\label{constant 2: C_1(d,p)}
\begin{aligned}
& C_1(d,\alpha,p) = \mathcal{A}(d,-\alpha) \int_1^{\infty} (s^p-1)(1-s^{-d+\alpha- p})s(s^2-1)^{-1-\alpha} H(s,d)ds, \\
& H(s,d) = 2\pi \frac{\pi^{(d-3)/2}}{\Gamma(\frac{d-1}{2})} \int_0^{\pi} \sin^{d-2}\theta \frac{(\sqrt{s^2 - \sin^2 \theta} + \cos \theta)^{1+\alpha}}{\sqrt{s^2 - \sin^2 \theta}}d\theta. 
\end{aligned}
\end{equation}
\end{prop}

\begin{remark} \label{constant: relation}
It is shown in \cite[Section 3.1]{CKSV20} that $C_0(d,\alpha,p)$ is strictly increasing on $(\frac{\alpha-1}{2},\alpha)$ and $$
C_0(d,\alpha,\alpha -1) = 0\ and \ \lim_{p \to \alpha -} C_0(d,\alpha,p) = \infty.$$
Similarly, by direct calculation, it is easy to see that $C_1(d,\alpha,p)$ is strictly increasing on $(0,\alpha)$ and $$
\lim_{p \to 0+}C_1(d,\alpha,p) = 0\ and \ \lim_{p \to \alpha-}C_1(d,\alpha,p) = \infty.$$
\end{remark}

\section{Perturbation arguments for $C^{1,\beta}$ open set}

In this section, we present some perturbation arguments for $C^{1,\be}$ open sets. 
In case $D$ is a $C^{1,1}$ open set, by using the analysis in \cite{BBC03}, one can construct appropriate super(sub)-harmonic functions based on $\delta_D^p(x)$.
The existence of inner and outer balls for $C^{1,1}$ open sets is crucial in this construction.
However, when $D$ is only assumed to be $C^{1,\be}$ with $\be<1$, 
inner and outer ball conditions do not hold. To get around this difficulty, we construct super(sub)-harmonic functions 
based the ``vertical'' distance to the boundary. 
This construction is inspired by \cite{Guan07}.

\subsection{Local analysis of $C^{1,\beta}$ open set}
We always assume that $D$ is a $C^{1, \beta}$ open set unless explicitly specified other wise. For any $p>0$ and $r>0$, we define
$\mathcal{P}^{r,\beta}: = \{y= (\widetilde{y},y_d) \in \mathbb{R}^{d-1} \times \mathbb{R}: r^{-1} |\widetilde{y}|^{1+\beta} < y_d < r\}$.

First we recall the key idea in the $C^{1,1}$ case. The following fact makes use of the inner ball and outer ball condition and is contained in the proofs of Lemmas 5.6, 5.9, 6.3 in \cite{BBC03}:
\begin{lemma}\label{C11 case}
Assume $D$ is a $C^{1,1}$ open set, $z \in \partial D$ and $CS_z$ is the local coordinate system based at $z$ in Definition
\ref{d:c1beta}. 
There exists $r_0>0$ depending only on the characterstics of $D$ such that for $p>0$, $r\in (0, r_0)$ and $y \in \big(\mathcal{P}^{r,1} \cup (- \mathcal{P}^{r,1}) \big)\cap CS_z$,  it holds that
\begin{equation} \label{difference of distances: C11}
\big| \delta_D(y)^p - |y_d|^p \big| \le c(p) |y_d|^{p-1}|\widetilde{y}|^2.
\end{equation}
\end{lemma}

\begin{center}
\begin{figure}[H] 
    \centering
    \begin{tikzpicture}
        \draw[->] (-2, 0) -- (2, 0) node[right] {$x$};
        \draw[->] (0, -0.5) -- (0, 2) node[above] {$y$};

        \draw[thick, domain=-1.5:1.5, smooth, samples=100] 
            plot (\x, {2*abs(\x)^(1.5)});

        \fill[blue!30, opacity=0.5] (0, 0.5) circle (0.5);
        \draw[thick, blue] (0, 0.5) circle (0.5);
    \end{tikzpicture}
    \caption{Failure of the inner ball condition for $C^{1,\beta}$ open set.}
    \label{figure:inner ball}
\end{figure}
\end{center}
However, for a $C^{1,\beta}$ open set with $\beta \in (0,1)$, \eqref{difference of distances: C11} 
may not be true. Specifically, the analog inequality 
\begin{align}\label{e:cex}
\big| \delta_D(y)^p - |y_d|^p \big| \le c(p) |y_d|^{p-1}|\widetilde{y}|^{1+\beta}
\end{align}
may not be correct. 
To illustrate this, suppose $d=2$ and $\psi$ is a function with a $C^{1, \beta}$ function on $\R$ which agrees with
the function $|x_1|^{1+\beta}$ when $|x_1|<1$. Let 
\[
D = \{(y_1, y_2) : y_2 > \psi(y_1)\}.
\]
For the point $(0, y_2)$ with $y_2$ a small positive number, we have
\begin{equation}\label{ineqn:C1beta vs C11}
    \delta_D((0, y_2)) < |y_2|.
\end{equation}
Then the left-hand side of \eqref{e:cex} is strictly positive, while the right-hand side of the inequality is zero.

The main reason is that, when $\beta < 1$, any ball centered at $(0, r)$ with (small enough) radius $r$ intersects the curve $y_2 = |y_1|^{1+\beta}$, see Figure \ref{figure:inner ball} for an illustration.

Therefore, when $D$ is $C^{1,\beta}$ for some $\beta < 1$, $\delta_D(\cdot)$ is no longer a suitable function to work with, since the crucial estimate \eqref{difference of distances: C11} does not hold. An alternative is to replace 
$\delta_D(x)$ by the vertical distance function.
Given $z \in \partial D$, we have
\[
D \cap B(z, R_0) = \{y \in \mathbb{R}^d : y_d > \psi_z(\widetilde{y})\} \cap B(z, R_0),
\]
where $R_0>0$ is localization radius of $D$ and $\psi_z$ is the $C^{1, \beta}$ function based at $z$ in Definition \ref{d:c1beta}.
The vertical distance function is then defined as:
\begin{equation}\label{distance-fcn-new: codimension 1}
    h_z(y) = |y_d - \psi_z(\widetilde{y})|, \quad \forall y \in CS_z.
\end{equation}
Note that $h_z$ is 
defined in terms of  $\psi_z$ and so it depends on $z$. 
The following lemma, which is valid when $D$ is Lipschitz, see \cite{Ancona78, Guan07}, says that, locally, the vertical distance $h_z(\cdot)$ is comparable to $\delta_D(\cdot)$. We give an elementary proof here for the reader's convenience.  
 
\begin{lemma} \label{comparison: codimension 1}
Let $R_0>0$ be the localization radius of $D$. 
For any $z \in \partial D$ and $y \in B(z,R_0/2) \cap D$, 
it holds that  $h_z(y) \asymp \delta_D(y)$ with comparison constants depending only on the characteristics of $D$.
If $D = \ZR^d_+$, we have $h_z(y) = \delta_D(y)$.
\end{lemma}
\begin{proof}
Fix $z\in \partial D$ and let $\psi=\psi_z$ be the $C^{1, \beta}$ function based at $z$ in Definition \ref{d:c1beta}. 
Since $\psi$ is Lipschitz with $\|\nabla \psi\|_{\infty} \le \Lambda$, there exists 
$k>0$ such that the cone 
$\mc C_{\psi(\wt y)} : = \{x\in B(z,R_0): x_d >\Lambda |\wt x - \wt y| + \psi_z(\wt y)\}$, 
with vertex at $(\wt y, \psi(\wt y)) \in \partial D$, 
is strictly contained in $D \cap B(z,R_0)$. 
Therefore, 
$$
\delta_D(y) \ge \text{dist}(y, \mc C_{\psi(\wt y)}) \wedge \frac{R_0}{2}  = \frac{h_z(y)}{\sqrt{\Lambda^2 + 1 }} \wedge \frac{R_0}{2},
$$ 
which concludes the proof since $\delta_D(y) \le h_z(y) \lesssim R_0$.
\end{proof}

Now we present our perturbation argument for $C^{1,\beta}$ open sets.
We start with special $C^{1,\beta}$ open sets defined as follows. Let $\psi: \mb{R}^{d-1} \to \mb{R}$ function with
characteristic $\Lambda$. The set $D_{\psi} := \{y = (\wt{y}, y_d) \in \mb{R}^{d-1} \times \mb{R}: y_d > \psi(\wt{y})\}$
is called a special $C^{1,\beta}$ open set.
Without loss of generality, we assume $\psi(0)=0, \nabla \psi(0) = 0$. 

For any $x = (\wt x, x_d)\in\R^d$, 
\begin{itemize}
\item define 
$x^{(0)} := (\wt{x}, \psi(\wt{x})) \in \partial D_{\psi}$. 
\item define $\Pi_x$ by the tangent plane to $\partial D_{\psi}$ at 
$x^{(0)}$. $\Pi_x$ is spanned by 

\begin{equation}\label{def:tangent plane definition}
\begin{aligned}
& v_1 = \{1,0,\cdots, 0, \frac{\partial \psi}{\partial x_1}(\wt{x})\} \\
& \cdots \\
& v_{d-1} = \{0,0,\cdots, 1, \frac{\partial \psi}{\partial x_{d-1}}(\wt{x})\}.
\end{aligned}
\end{equation}

\item define the vertical distance 
$h(x) = |x-x^{(0)}| = |x_d - \psi(\wt x)|$.
\item define $h_x^*(y)$ to be the distance between $y \in \mb R^d$ and the tangent space $\Pi_x$: $h_x^*(y) := \text{dist}(y,\Pi_x)$.
\end{itemize}

\begin{center}
\begin{figure}[H]
\begin{tikzpicture}
\draw [->] (0,0) -- (10,0) node [right] {$\wt x$};
\node at (-.5,0) {$O$};
\draw[scale=1, domain=0:10, smooth, variable=\x, blue] plot ({\x}, {\x*sqrt(\x)/10});
\fill (0,0) circle (.3mm);
\fill (4,.8) circle (.3mm);
\fill (5.11,1.07) circle (.3mm);
\fill (5.31,1.23) circle (.3mm);
\fill (4,5) circle (.3mm);
\node at (4,.5) {$x^{(0)}$};
\node at (5.4,.9) {$x^{(1)}$};
\node at (5.6,1.55) {$x^{(2)}$};
\node at (2,5) {$h(x) = |x-x^{(0)}|$};
\node at (2,4) {$h_x^*(x) = |x-x_1|$};
\node at (2,3) {$\delta_{D_\psi}(x) = |x-x_2|$};
\node at (4,5.2) {$x$};

\draw[black,dotted] (4,.8 ) -- (4,5);
\draw[black,dotted] (4,5) -- (5.1,1.1);
\draw[black,dotted] (4,5) -- (5.31,1.23);
\draw [->] [black] (4,.8) -- (9,2);
\node at (9.3,3.3) {$\psi(\wt x)$};
\node at (9.3,1.85) {$\Pi_x$};
\node at (8.3,1.5) {$\psi^*(\wt x)$};
\end{tikzpicture}
\caption{Illustration of $h(x),h_x^*(x), \Pi_x$.}
\label{figure1}
\end{figure}
\end{center}
We illustrate the above definition in Figure \ref{figure1}. Using Lemma \ref{comparison: codimension 1}, for any $x\in D_{\psi}$ with $x$ close to the boundary, 
\begin{equation}\label{equivalence of three distances}
    h_x^*(x) \asymp h(x) \asymp \delta_{D_{\psi}}(x).
\end{equation}
With a more detailed analysis, we can get the following finer result. Note that the assumption $|x|<r_1$ can be relaxed to $\delta_{D_{\psi}}(x) < r_1$.  We only state this special form since this is the form that will be used later.
\begin{lemma} \label{key lemma: distance comparison}
Assume $D_{\psi}$ is a special $C^{1,\beta}$ open set defined as above. There exist positive constants $r_1$ and $A_0$ depending only on the characteristic of $D_\psi$ such that for all $x \in D_{\psi}$ with $|x|<r_1$, it holds that
\begin{equation}
|h_x^*(x) - \delta_{D_{\psi}}(x)| \le A_0 \delta_{D_{\psi}}(x)^{1+\be}.
\end{equation}
\end{lemma}
\begin{proof}
Fix an arbitrary $x \in D_{\psi}$ and set $x^{(0)} = (\widetilde{x}, \psi(\widetilde{x}))$. 
We introduce a new coordinate system, 
denoted by $CS_{x^{(0)}}^*$, with origin at $x_0$ such that $\Pi_x = \{ y^* \in CS_{x^{(0)}}^* : y_d^* = 0 \}$.
In $CS_{x^{(0)}}^*$, 
the $C^{1,\beta}$ function $\psi$ is represented by another $C^{1,\beta}$ function $\widehat{\psi}$ 
such that $\widehat{\psi}(\widetilde{0}^*) = 0$ and $\nabla \widehat{\psi}(\widetilde{0}^*) = 0$. See Figure \ref{figure2} for an illustration of this new coordinate system.
\begin{center}
\begin{figure}[H]
\begin{tikzpicture}
\draw [->] (0,0) -- (8,0) node [right] {$\Pi_x$};
\node at (-.5,0) {$x^{(0)}$};
\node at (2,4.2) {$x = (\wt{x}^*, x_d^*)$};
\node at (2,-0.2) {$x^{(1)}$};
\node at (1.8,0.7) {$x^{(3)}$};
\node at (4.9,1.25) {$x^{(2)}= (\wt y^*, \widehat{\psi}(\wt y^*))$};
\node at (4.8,2.6) {$\widehat{\psi}$};
\draw[scale=1, domain=0:8, smooth, variable=\x, blue] plot ({\x}, {\x*sqrt(\x)/5});
\fill (0,0) circle (.3mm);
\fill (2,4) circle (.3mm);
\fill (2,0.565) circle (.3mm);
\fill (2,0) circle (.3mm);
\fill (3.8,1.48) circle (.3mm);
\draw[black,dotted] (2,4) -- (2,0);
\draw[black,dotted] (2,4) -- (3.8,1.48);
\end{tikzpicture}
\caption{Illustration of the new coordinate system $CS_{x^{(0)}}^*$.}
\label{figure2}
\end{figure}
\end{center}
Now we compare $h_x^*(x)$ and $\delta_{D_{\psi}}(x)$. First, we express $h_x^*(x)$ and $\delta_{D_{\psi}}(x)$ in 
the $CS_{x_0}^*$. In $CS_{x^{(0)}}^*$, we write
\begin{equation}
x = (\wt{x}^*, x_d^*),\quad x^{(1)} = (\wt{x}^*, 0),\quad x^{(2)} = (\wt y^*,y_d^*) = (\wt y^*, \widehat{\psi}(\wt y^*)).
\end{equation}
Then by definition,
\begin{equation}
    h_x^*(x) = |x-x^{(1)}| = |x_d^*|,\quad \delta_{D_{\psi}}(x) = |x-x^{(2)}| = \sqrt{|\wt y^* - \wt x^*|^2 + (x_d^* - \widehat{\psi}(\wt y^*))^2}.
\end{equation}
Note that the distances are equivalent (see \eqref{equivalence of three distances}) 
when $|x|$ is small. 
Thus we can choose $r_1$ small enough so that, whenever \( |x| < r_1 \), the following holds:
\begin{equation}\label{e:elemfacts}
    \sqrt{|\wt x^*|^2 + |x_d^*|^2} \asymp \delta_{D_{\psi}}(x) \asymp |x_d^*|, \quad |\wt x^*| \lesssim |x_d^*|,
\end{equation}
where the implicit constants depend only on the characteristic of $\psi$. It remains to show that
\begin{equation}\label{geometry lemma:goal}
    \big| |x_d^*| - \sqrt{|\wt y^* - \wt x^*|^2 + (x_d^* - \widehat{\psi}(\wt y^*))^2} \big| \lesssim |x_d^*|^{1+\beta},
\end{equation}
which is equivalent to that there exist constants $c_1, c_1'>0$ depending on the characteristic of $\psi$ such that
\begin{align}\label{geometry lemma:goal equiv}
 |x_d^*| - c_1 |x_d^*|^{1+\beta} \le \sqrt{|\wt y^* - \wt x^*|^2 + (x_d^* - \widehat{\psi}(\wt y^*))^2} \le  |x_d^*| + c_1' |x_d^*|^{1+\beta}.
\end{align}
The upper bound is simple. In fact, since $|x - x^{(2)}| = \sqrt{|\wt y^* - \wt x^*|^2 + (x_d^* - \widehat{\psi}(\wt y^*))^2}$, 
choosing $x^{(3)} = (\wt x^*,\widehat{\psi}(\wt x^*))$ and using the triangle inequality, we get
\begin{align*}
    |x-x^{(2)}| \le |x - x^{(3)}| & \le |x - x^{(1)}| + |x_1 - x^{(3)}| \\
    & = |x_d^*| + |\widehat{\psi}(\wt x^*)| \le |x_d^*| + 
    \Lambda |\wt x^*|^{1+ \beta}\lesssim|x^*_d|^{1+\beta}.
\end{align*}
In the last inequality, we used the last relation in \eqref{e:elemfacts}. In the penultimate inequality, we used the facts $\widehat{\psi}(\wt 0^*) = 0, \nabla \widehat{\psi}(\wt 0^*) = 0$ and the mean value theorem
\begin{align}\label{geometry lemma:intermediate mean value}
    |\widehat{\psi}(\wt x^*)| 
    = |\nabla \widehat{\psi}(\wt \xi^*)| |\wt x^*| = |\nabla \widehat{\psi}(\wt \xi^*) - \nabla \widehat{\psi}(\wt 0^*)| |\wt x^*| \le 
    \Lambda |\wt x^*|^{1+\beta},
\end{align}
where $\wt \xi^*$ is a point on the line segment connecting $\wt 0^*, \wt x^*$. 

It remains to prove the lower bound in 
\eqref{geometry lemma:goal equiv}, which is equivalent to 
\begin{align*}
    |\wt y^* - \wt x^*|^2 + (x_d^* - \widehat{\psi}(\wt y^*))^2 & \ge |x_d^*|^2 + c_1 |x_d^*|^{2+2\beta} - 2c_1|x_d^*|^{2+\beta} \\
    & \ge |x_d^*|^2 - c_1|x_d^*|^{2+\beta} + c_1|x_d^*|^{2+\beta} (|x_d^*|^\be - 1).
\end{align*}
By choosing $r_1>0$ smaller if necessary, we can assume $|x_d^*|^\be < 1$.  Therefore, to show the above inequality, it suffices to show
\begin{equation}\label{geometry lemma:intermediate 1}
     |\wt y^* - \wt x^*|^2  + \widehat{\psi}(\wt y^*)^2 + c_1|x_d^*|^{2+\beta} \ge 2 x_d^* \widehat{\psi}(\wt y^*). 
\end{equation} 
We now prove \eqref{geometry lemma:intermediate 1} in two separate cases.

\textbf{Case I}: $|\wt y^* - \wt x^*| \ge |x_d^*|$. In this case, 
\begin{align*}
     2 x_d^* \widehat{\psi}(\wt y^*) \le |x_d^*|^2 +  
     \widehat{\psi}(\wt y^*)^2.
\end{align*}

\textbf{Case II}: $|\wt y^* - \wt x^*| \le |x_d^*|$. In this case, we have $|\wt y^*|  \le  |\wt x^*| + |x_d^*| \lesssim |x_d^*|$. Then similar to \eqref{geometry lemma:intermediate mean value}, we have 
\begin{equation}\label{geometry lemma:intermediate 2}
    |\widehat{\psi}(\wt y^*)| \le \Lambda |\wt y^*|^{1+\beta} \le c_2 |x_d^*|^{1+\beta}.
\end{equation}
Combining this with the last relation in \eqref{e:elemfacts}, we get
\begin{align*}
     2 x_d^* \widehat{\psi}(\wt y^*) \le c_3 |x_d^*|^{2+\beta}. 
\end{align*}
Therefore, by choosing $c_1 > c_3$, 
we arrive at \eqref{geometry lemma:intermediate 1}. Hence \eqref{geometry lemma:goal} is valid.
\end{proof}

In the remainder of this paper, $r_1$ always stands for the constant in Lemma \ref{key lemma: distance comparison}. 

\subsection{Perturbation argument for $C^{1,\beta}$ special open set.}
Now we prove the main result of this subsection, which is a perturbation argument near the origin, using the vertical distance as a test function.
\begin{prop} \label{key prop: perturbation}
Suppose $D_{\psi}$ is a special $C^{1,\beta}$ open set and $h(y) := |y_d - \psi(\wt y)|, y \in \mb R^d$, is the vertical distance function. Then for any $p\in [\alpha-1,\alpha)\cap(0,\alpha)$, 
there exists $A_1>0$ depending on $p$ and the characteristic of $D_\psi$ 
such that for all $x\in D_{\psi}$ with $|x|<r_1$, 
\begin{align}\label{key-esti}
&\big|  \mc{A}(d,-\alpha){\ \rm p.v.} \int_{D_{\psi}} \frac{h(y)^p - h(x)^p}{|y-x|^{d+\alpha}}dy - C_0(d,\alpha,p)h(x)^p \delta_{D_{\psi}}(x)^{-\alpha} \big | \\
&\le  A_1(1+ \delta_{D_{\psi}}(x)^{p-\alpha+\ga} + |\log(\delta_{D_{\psi}}(x))|),\nonumber
\end{align}
where $C_0(d,\alpha,p)$ is defined in \eqref{constant 1: C_0(d,p)} and $\ga = \min\{\beta, \beta \frac{1+\alpha}{1+\beta}\}$.
\end{prop}
\begin{proof}
We start with an overview of the proof.  We first use Proposition \ref{harmonic: 1 codimension} to rewrite the left hand side  as the sum of two terms, one of them involves the killing function. For the term involving the killing function, we apply Lemma \ref{key lemma: distance comparison}. For the other term, we refine certain known techniques (e.g., \cite{BBC03, Guan07, CKSV20} to achieve the desired result.

We fix $x \in D_{\psi}$ with $|x|<r_1$.

\textbf{Step I}:
Recall the definition \( h_x^*(y) = \text{dist}(y, \Pi_x) \) in \eqref{def:tangent plane definition}, with \( \Pi_x \) denoting the tangent space of \( \psi \) at \( x^{(0)} = (\wt x, \psi(\wt x)) \). Then 
 $$
 h(x) = \sqrt{1+|\nabla \psi(\wt{x})|^2} h_x^*(x).
 $$
Define $\psi^*(\wt y): = \nabla \psi(\wt x) \cdot (\wt y - \wt x) + \psi(\wt x)$. $y_d=\psi^*(\wt y)$ is the equation of the tangent plane $\Pi_x$.
By Proposition \ref{harmonic: 1 codimension}, 
\begin{align*}
\mc L_{\alpha}^{D_{\psi^*}} (h_x^{*p})(x) := 
\mc{A}(d,-\alpha) {\ \rm p.v.}\int_{D_{\psi^*}} \frac{h_x^*(y)^p - h_x^*(x)^p}{|y-x|^{d+\alpha}}dy - C_0(d,\alpha,p)h_x^*(x)^{p-\alpha} = 0.
\end{align*}
Thus we have 
\begin{align*}
& \mc{A}(d,-\alpha) {\ \rm p.v.}\int_{D_{\psi}} \frac{h(y)^p - h(x)^p}{|y-x|^{d+\alpha}}dy - C_0(d,\alpha,p)h(x)^p \delta_{D_{\psi}}(x)^{-\alpha} \\
& = \mc{A}(d,-\alpha) {\ \rm p.v.}\int_{D_{\psi}} \frac{h(y)^p - h(x)^p}{|y-x|^{d+\alpha}}dy - C_0(d,\alpha,p)h(x)^p \delta_{D_{\psi}}(x)^{-\alpha} \\
& - \left(\sqrt{1+|\nabla \psi(\wt{x})|^2}\right)^p \mc L_{\alpha}^{D_{\psi^*}} (h_x^{*p})(x) \\
& = \mc{A}(d,-\alpha) \left({\ \rm p.v.}\int_{D_{\psi}} \frac{h(y)^p - h(x)^p}{|y-x|^{d+\alpha}}dy - {\ \rm p.v.} \int_{D_{\psi^*}} \frac{(h_x^*(y)\sqrt{1+|\nabla \psi(\wt x)|^2})^p - h(x)^p}{|y-x|^{d+\alpha}}dy \right) \\
& + C_0(d,\alpha,p)h(x)^p(h_x^*(x)^{-\alpha} - \delta_{D_{\psi}}(x)
^{-\alpha})\\
&=:\mc{A}(d,-\alpha) I+C_0(d,\alpha,p) II.
\end{align*}
To estimate $|II|$, we note that $h(x)\asymp h_x^\ast(x)\asymp \de_{D_{\psi}}(x)$ with comparison constants depending on the characteristics of $\psi$. 
 It follows from  Lemma \ref{key lemma: distance comparison} that
\begin{equation}
\label{reduction-1}
    |h_x^*(x) - \delta_{D_{\psi}}(x)|\lesssim  \delta_{D_{\psi}}(x)^{\beta+1}.
\end{equation}
Applying the mean-value theorem to the function $x^{-\al}$, we get
\begin{equation}\label{part-II}
    |II|= h(x)^p \left| h_x^*(x)^{-\alpha} - \delta_{D_{\psi}}(x)^{-\alpha} \right| \lesssim \delta_{D_{\psi}}(x)^p \delta_{D_{\psi}}(x)^{-\alpha - 1} \delta_{D_{\psi}}(x)^{1+\be} = \delta_{D_{\psi}}(x)^{p-\alpha+\beta}.
\end{equation}

\noindent \textbf{Step II}: In this step, we start to estimate $|I|$. Note that
\begin{equation}
\begin{aligned}
    I &: = \text{p.v.}\int_{D_{\psi}} \frac{h(y)^p - h(x)^p}{|y-x|^{d+\alpha}}dy - \ppv \int_{D_{\psi^*}} \frac{(h_x^*(y)\sqrt{1+|\nabla \psi(\wt x)|^2})^p - h(x)^p}{|y-x|^{d+\alpha}}dy \\
    & = \ppv \int_{D_{\psi} \cap D_{\psi^*}} \frac{h(y)^p - (h_x^*(y)\sqrt{1+|\nabla \psi(\wt x)|^2})^p}{|y-x|^{d+\alpha}}dy \\
    & + \text{p.v.}\int_{D_{\psi}\backslash D_{\psi^*}} \frac{h(y)^p - h(x)^p}{|y-x|^{d+\alpha}}dy - \ppv \int_{D_{\psi^*}\backslash D_{\psi}} \frac{(h_x^*(y)\sqrt{1+|\nabla \psi(\wt x)|^2})^p - h(x)^p}{|y-x|^{d+\alpha}}dy.
\end{aligned}
\end{equation}
To estimate $|I|$, we define
\begin{align*}
& U_1:= \{y\in \mb R^d: |\wt y - \wt x| < r_1, y_d < 2r_1\}, \\
& U_2:= \{y \in \mb R^d: |\wt y - \wt x| < r_1, \psi(\wt y)< y_d < \psi^*(\wt y)\quad or\quad  \psi^*(\wt y)< y_d < \psi(\wt y)\}.
\end{align*}
Since $|x|< r_1$, we have $B(x,r_1) \cap D_{\psi} \subseteq U_1 \cap D_{\psi}$. Thus
\begin{align*}
    & D_{\psi} \cap D_{\psi^*} \subseteq (D_{\psi} \cap D_{\psi^*} \cap U_1) \bigcup B(x,r_1)^c, \\
    & (D_{\psi^*}\backslash D_{\psi}) \bigcup (D_{\psi}\backslash D_{\psi^*}) \subseteq \overline{U_2} \bigcup B(x,r_1)^c.
\end{align*}
\begin{center}
\begin{figure}
\begin{tikzpicture}
\draw [->] (-3,0) -- (6,0) node [right] {$\wt x$};
\draw[scale=1, domain=-3:6, smooth, variable=\x, blue] plot ({\x}, {abs(\x)^(1.8)/5});
    \pgfmathsetmacro{\fx}{abs(2)^(1.8)/5}    
    \pgfmathsetmacro{\fpx}{1.8*abs(2)^(0.8)/5} 
\draw[black, domain=-3:6, smooth, variable=\x] plot ({\x}, {\fpx * (\x - 2) + \fx});
    \begin{scope}
        \clip (-1,-2) rectangle (5,5);
        \fill[blue!20, opacity=0.5] 
            plot[domain=-1:5, smooth, variable=\x] ({\x}, {abs(\x)^(1.8)/5})
            -- plot[domain=5:-1, smooth, variable=\x] ({\x}, {\fpx * (\x - 2) + \fx});
    \end{scope}

        \begin{scope}
        \clip (-1,0) rectangle (5,6);
        \fill[red!20, opacity=0.5] 
            plot[domain=-1:5, smooth, variable=\x] ({\x}, {abs(\x)^(1.8)/5})
        -- plot[domain=5:-1, smooth, variable=\x] ({\x}, {5.5}) -- cycle;
    \end{scope}
\fill (0,0) circle (.3mm);
\fill (2,2) circle (.3mm);
\fill (2,0.7) circle (.3mm);
\fill[black!30, opacity=0] (2, 2) circle (3);
        \draw[dotted, blue] (2, 2) circle (3);
\draw[thick] (-1,-3) -- (-1,5.5);
\draw[thick] (5,-3) -- (5,5.5);
\draw[thick] (-1,5.5) -- (5,5.5);
\node at (0,0.35) {$O$};
\node at (1.7,2) {$x$};
\node at (2,4) {$D_{\psi} \cap D_{\psi^*} \cap U_1$};
\node at (-1.5,-0.35) {$U_2$};
\node at (-1.5,0.8) {$\psi$};
\node at (-1.5,-2) {$\psi^*$};
\node at (5.9,2.8) {$\Pi_x$};
\end{tikzpicture}
\caption{Illustration of $D_{\psi} \cap D_{\psi^*} \cap U_1$(red) and $U_2$(blue).}
\label{figure:partition}
\end{figure}
\end{center}
Hence,
\begin{align*}
|I| & \le \int_{D_{\psi} \cap D_{\psi^*} \cap U_1} \frac{|h(y)^p - (h_x^*(y)\sqrt{1+|\nabla \psi(\wt x)|^2})^p|}{|y-x|^{d+\alpha}}dy \\
& + \int_{U_2} \frac{|(h_x^*(y)\sqrt{1+|\nabla \psi(\wt x)|^2})^p - h(x)^p|}{|y-x|^{d+\alpha}}dy+ \int_{U_2} \frac{|h(y)^p - h(x)^p|}{|y-x|^{d+\alpha}}dy \\
& + \int_{B(x,r_1)^c}\frac{|(h_x^*(y)\sqrt{1+|\nabla \psi(\wt x)|^2})^p - h(x)^p|}{|y-x|^{d+\alpha}}dy + \int_{B(x,r_1)^c}\frac{|h(y)^p - h(x)^p|}{|y-x|^{d+\alpha}}dy \\
& =: I_1 + I_2 + I_3 + I_4 + I_5.
\end{align*}
$I_4$ and $I_5$ can be bounded by constants. 
We will show  $I_5 \lesssim r_1^{p-\alpha}$. $I_4$ can dealt with similarly.
Note that $|h(y)- h(x)| = \big|y_d - \psi(\wt y) - (x_d - \psi(\wt x))\big| \le (\Lambda + 1) |y - x|$, thus we have 
\begin{align*}
 |h(y)^p- h(x)^p|& \le h(y)^p + h(x)^p \le \left(h(x) + (1+\Lambda) |y-x|\right)^p + h(x)^p \\
 & \le \max\{2^{p-1}, 1\}\left(h(x)^p + (1+\Lambda)^p |y-x|^p\right) + h(x)^p \\
 & \lesssim h(x)^p + |y-x|^p.
\end{align*}
This implies that
\begin{align}\label{non-singular argument}
\int_{B(x,r_1)^c}\frac{|h(y)^p - h(x)^p|}{|y-x|^{d+\alpha}}dy \lesssim  \int_{B(x,r_1)^c} \frac{h(x)^p dy}{|y-x|^{d+\alpha}} + \int_{B(x,r_1)^c}\frac{|y-x|^p}{|y-x|^{d+\alpha}}dy\lesssim r_1^{p-\alpha}.
\end{align}

Now we estimate $I_1$. By
the  mean-value theorem, for $a,b>0$, we can easily see that 
$|a^p-b^p|  \le p(a^{p-1}+b^{p-1})|a-b|$. Thus,
\begin{align}\label{e:I1}
I_1 & = \int_{D_{\psi} \cap D_{\psi^*} \cap U_1} \frac{|h(y)^p - (h_x^*(y)\sqrt{1+|\nabla \psi(\wt x)|^2})^p|}{|y-x|^{d+\alpha}}dy \\
&\lesssim
\int_{D_{\psi} \cap D_{\psi^*} \cap U_1} \frac{(h_x^*(y)\sqrt{1+|\nabla \psi(\wt x)|^2})^{p-1}|h(y)- h_x^*(y)\sqrt{1+|\nabla \psi(\wt x)|^2}|}{|y-x|^{d+\alpha}}dy\nonumber\\
& + \int_{D_{\psi} \cap D_{\psi^*} \cap U_1} \frac{h(y)^{p-1}|h(y)- h_x^*(y)\sqrt{1+|\nabla \psi(\wt x)|^2}|}{|y-x|^{d+\alpha}}dy \nonumber\\
& = \int_{D_{\psi} \cap D_{\psi^*} \cap U_1} \frac{(y_d - \psi^*(\wt y))^{p-1}|h(y)- h_x^*(y)\sqrt{1+|\nabla \psi(\wt x)|^2}|}{|y-x|^{d+\alpha}}dy \nonumber\\
& + \int_{D_{\psi} \cap D_{\psi^*} \cap U_1} \frac{(y_d -\psi(\wt y))^{p-1}|h(y)- h_x^*(y)\sqrt{1+|\nabla \psi(\wt x)|^2}|}{|y-x|^{d+\alpha}}dy \nonumber\\
& \lesssim \int_{D_{\psi} \cap D_{\psi^*} \cap U_1} \frac{\big(|y_d - \psi^*(\wt y)|^{p-1} + |y_d - \psi(\wt y)|^{p-1} \big)|\wt y - \wt x|^{1+\beta}}{|y-x|^{d+\alpha}}dy,\nonumber
\end{align}
where in the last equality, we used the facts that $$h(y) = y_d -\psi(\wt y), \ \sqrt{1+|\nabla \psi(\wt x)|^2} h_x^*(y)= y_d - \psi^*(\wt y),$$ and 
in the last inequality of \eqref{e:I1}, we used 
the mean value theorem:
\begin{align*}
& |h(y)- h_x^*(y)\sqrt{1+|\nabla \psi(\wt x)|^2}| = \big|(y_d -\psi(\wt y)) - (y_d - \psi^*(\wt y))\big| \\
& = |\psi(\wt y) -  \psi^*(\wt y)| \lesssim |\wt y - \wt x|^{1+\be}.
\end{align*}
Since $\psi^*, \psi$ are Lipschitz with Lipschitz constant $\Lambda$, for any $y \in D_{\psi} \cap D_{\psi^*} \cap U_1$, using the facts $|\wt y - \wt x|<r_1$ and $|\wt x|<r_1$, we get
\begin{align}\label{e:pre4chv}
   \max\{ |y_d - \psi(\wt y)|, |y_d - \psi^*(\wt y)|\} \le 2r_1 + 2\Lambda r_1. 
\end{align}
Define
\begin{align}\label{eqn:Psi change of variable}
\Psi: (\wt y, y_d)\mapsto (\wt y, y_d - \psi(\wt y)),\quad \Psi^*: (\wt y, y_d)\mapsto (\wt y, y_d - \psi^*(\wt y)).
\end{align}
Using \eqref{e:pre4chv} we get
\begin{equation} \label{e:s1}
    \begin{aligned}
& \Psi^*(D_{\psi} \cap D_{\psi^*} \cap U_1) \bigcup \Psi(D_{\psi} \cap D_{\psi^*} \cap U_1) \\
& \subseteq \{y: |\wt y -\wt x|<r_1, 0< y_d < 2r_1 + 2\Lambda r_1\}=: \wt{U_1}.
\end{aligned}
\end{equation}
Using the Lipschitz property of $\Psi, \Psi^*$, the changes of variables $z = \Psi^*(y),\ z = \Psi(y)$ separately, and \eqref{e:s1}, we get
\begin{align*}
& \int_{D_{\psi} \cap D_{\psi^*} \cap U_1} \frac{\big(|y_d - \psi^*(\wt y)|^{p-1} + |y_d - \psi(\wt y)|^{p-1} \big)|\wt y - \wt x|^{1+\beta}}{|y-x|^{d+\alpha}}dy \\
& \le c_1\int_{D_{\psi} \cap D_{\psi^*} \cap U_1} \frac{(y_d - \psi^*(\wt y))^{p-1}|\wt y - \wt x|^{1+\beta}}{|\Psi^*(y)-\Psi^*(x)|^{d+\alpha}}dy  \\
& + c_1 \int_{D_{\psi} \cap D_{\psi^*} \cap U_1} \frac{(y_d - \psi(\wt y))^{p-1}|\wt y - \wt x|^{1+\beta}}{|\Psi(y)-\Psi(x)|^{d+\alpha}}dy  \\
& \le c_2 \int_{\wt U_1} \frac{|z_d|^{p-1} |\wt z - \wt x|^{1+\beta}}{(\sqrt{|\wt z - \wt x|^2 + |z_d - h(x)|^2})^{d+\alpha}} dz \\
& \le c_3 \left(1 + h(x)^{p-\alpha + \beta} + |\log h(x)| \right).
\end{align*}
The last inequality follows from a routine argument by splitting into two cases $|\wt z - \wt x| \ge |z_d - h(x)|$ and $|\wt z - \wt x| \le |z_d - h(x)|$. We omit the details and refe
Note that $\beta > (\alpha - 1)_+$ is needed here. 

\noindent \textbf{Step III}: In this step we estimate $I_2$ and $I_3$. Since the arguments are same, we only give the argument for $I_2$. 
First note that, if $y \in U_2$, then
\begin{equation}\label{U2 fact 1}
h(y) \lesssim  |\wt y-\wt x|^{1+\be},
\end{equation}
with implicit constant depending on $r_1$. In fact, since  $y \in U_2$ means $\psi(\wt y) < y_d < \psi^*(\wt y)$ or $\psi^*(\wt y) < y_d <\psi(\wt y)$, we have 
\begin{align*}
h(y) = |y_d -\psi(\wt y)| \le |\psi(\wt y) - \psi^*(\wt y)| = |\psi(\wt y) -\psi(\wt x) - \nabla \psi(\wt x)\cdot (\wt y - \wt x)| \le c |\wt y - \wt x|^{1+\be}.
\end{align*}
For the last inequality, we used the fact that $\psi$ is a $C^{1,\be}$ function.  

We claim that 
\begin{equation}\label{U2 fact 2}
|y-x| \gtrsim |\wt y - \wt x| + h(x), \quad y \in U_2,
\end{equation}
with implicit constant depending on $r_1$. In fact, 
using \eqref{U2 fact 1} and the fact $|y-x| \ge |\wt y - \wt x|$, we get
\begin{align*}
h(x) & = |x_d -\psi(\wt x)| \le |x_d - y_d| + |y_d -\psi(\wt x)| \\
& \le |x-y| + h(y) + |\psi(\wt y) -\psi(\wt x)| \\
& \lesssim |x - y| + |\wt x - \wt y| \lesssim |x-y|,
\end{align*}
with implicit constant depending on $r_1$. Thus we proved that $|y-x| \gtrsim |\wt y - \wt x| + h(x)$. 

Let $m_r(\cdot)$ be the surface measure on the cylinder $|\wt y-\wt x|=r$. Then we have
\begin{align*}
&\frac{1}{(\sqrt{1 + |\nabla \psi(\wt x)|^2})^p}  I_2  = \int_0^{r_1} \int_{y \in U_2, |\wt y - \wt x| = r} \frac{|h_x^*(y)^p - h_x^*(x)^p|}{|y-x|^{d+\alpha}}
m_r(dy)dr \\
& = \left(\int_0^{h(x)^{1/(1+\beta)}} + \int_{h(x)^{1/(1+\beta)}}^{r_1}\right)\int_{y \in U_2, |\wt y - \wt x| = r} \frac{|h_x^*(y)^p - h_x^*(x)^p|}{|y-x|^{d+\alpha}}m_r(dy)dr \\
& =: I_{2,1} + I_{2,2}.
\end{align*}
Here, without loss of generality, we assumed $h(x)^{1/(1+\beta)} < r_1$, otherwise $I_{2,2}$ would not appear and it would not affect our result. 

We first deal with $I_{2,1}$. In this case, since $|\wt y - \wt x| =r \le h(x)^{\frac{1}{1+\be}}$, by \eqref{U2 fact 1}, we have $h_x^*(y) \lesssim |\wt y - \wt x|^{1+\beta} \le h(x)\asymp h_x^*(x)$. Thus using \eqref{U2 fact 2},
\begin{align*}
I_{2,1} &= \int_{0}^{h(x)^{1/(1+\be)}} \int_{y \in U_2, |\wt y - \wt x| = r} \frac{|h_x^*(y)^p - h_x^*(x)^p|}{|y-x|^{d+\alpha}}m_r(dy) dr \\
& \le c_4 \int_0^{h(x)^{1/(1+\be)}} \frac{h(x)^p}{(r + h(x))^{d+\alpha}}  \int_{y \in U_2, |\wt y - \wt x| = r}   m_r(dy) dr  \\
& \le c_5 \int_0^{h(x)^{1/ (1+\be)}}  \frac{r^{d+\be - 1} h(x)^p}{(r + h(x))^{d+\alpha}} dr  = c_5 h(x)^{p-\alpha + \beta} \int_0^{h(x)^{-\be/(1+\be)}} \frac{r^{d+\beta - 1}}{(r+1)^{d+\alpha}} dr \\
& \le h(x)^{p-\alpha + \be}\big(c_{6} + c_{6} \int_1^{h(x)^{-\be/(1+\be)}}r^{\be - \alpha -1}dr\big) \\
& \lesssim
\begin{cases}
1+ h(x)^{p-\alpha + \be},\ \be - \al < 0, \\
1+ h(x)^{p-\alpha + \be} + |\log(h(x))|,\ \be - \al= 0, \\
1 + h(x)^{p-\alpha + \beta \frac{1+\alpha}{1+\be}},\ \be - \al>0.
\end{cases}
\end{align*}
In the second inequality above, we used
\begin{align*}
\int_{U_2, |\wt y - \wt x|= r} m_r(dy) = \int_{|\wt y - \wt x| = r} |\psi(\wt y) - \psi^*(\wt y)| d\wt y \lesssim r^{1+\be} r^{d-2} = r^{d+\be - 1},
\end{align*}
which is due to 
$U_2 = \{y \in \mb R^d: |\wt y - \wt x| < r_1, \psi(\wt y)< y_d < \psi^*(\wt y)\quad or\quad  \psi^*(\wt y)< y_d <\psi(\wt y)\}$.

Now we estimate $I_{2,2}$. In this case, since $r \ge h(x)^{1/(1+\be)}$, we have $h(x) \le r^{1+\beta}$ and $h(y)\lesssim |\wt y - \wt x|^{1+\beta} = r^{1+\beta}$. Therefore,
\begin{align*}
I_{2,2} &= \int_{h(x)^{1/(1+\be)}}^{r_1} \int_{U_2, |\wt y - \wt x| = r} \frac{|h_x^*(y)^p - h_x^*(x)^p|}{|y-x|^{d+\alpha}}dy dr \\
& \le c_7 \int_{h(x)^{1/(1+\be)}}^{r_1} \int_{U_2, |\wt y - \wt x| = r} \frac{r^{p(1+\be)}}{(r+h(x))^{d+\alpha}}dy dr \\
& \le c_8 \int_{h(x)^{1/(1+\be)}}^{r_1}  \frac{r^{p(1+\be)} r^{d+\be -1}}{(r+h(x))^{d+\alpha}}dr = c_{8} h(x)^{p-\alpha + \be + p\beta}\int_{h(x)^{-\be/(1+\be)}}^{r_1/h(x)}  \frac{r^{p(1+\be)} r^{d+\be -1}}{(r+1)^{d+\alpha}}dr \\
& \le c_{9} h(x)^{p-\alpha + \be + p\beta}\int_{h(x)^{-\be/(1+\be)}}^{r_1/h(x)} r^{p- \alpha + \be + p \beta - 1}dr \\
& \lesssim 
\begin{cases}
h(x)^{p-\alpha + \be + p\be},\ p-\alpha + \be + p\be < 0, \\
|\log(h(x))|,\ p-\alpha + \be + p\be= 0, \\
1,\ p-\alpha + \be + p\be>0.
\end{cases}
\end{align*}

Similarly we have
$$
I_3\lesssim 
\begin{cases}
h(x)^{p-\alpha + \be + p\be},\ p-\alpha + \be + p\be < 0, \\
|\log(h(x))|,\ p-\alpha + \be + p\be= 0, \\
1,\ p-\alpha + \be + p\be>0.
\end{cases}
$$

In conclusion, we have shown that
\begin{equation}
|I| \lesssim 1 + h(x)^{p-\alpha + \gamma}+ | \log(h(x))|,
\end{equation}
where $\ga = \min\{\beta, \beta \frac{1+\alpha}{1+\beta}\}$. The proof is now complete.
\end{proof}

\subsection{Perturbation argument for the general case}

Recall that
$$
D_\kappa=D\setminus \cup _{l=1}^{m_1}\{x_l\},
$$
where $D$ is a $C^{1, \beta}$ open set with characteristics $(R_0, \Lambda)$,   $x_l, l=1, \dots, m_1,$ are distinct points in $D$ and $\partial D=\cup _{j=1}^{m_0} \Si_j$, with $\Si_j, j=1, \dots, m_0,$ being the connected components of $\partial D$. 
Define
\begin{align}\label{e:R1}
R_1= \min_{i\neq j}\{{\rm dist}(\Sigma_i, \Sigma_j)
\}\wedge \min_{k\neq l}\{|x_k-x_l|\}\wedge \min_l\{\delta_{D'}(x_l)\}\wedge R_0>0.
\end{align}

If $w \in \partial D_\kappa$, then
$w \in \Si_j$ for some  $j=1, \dots, m_0$ or $w=x_l$ for some $l=1, \dots, m_1$. We can define the truncated vertical distance function $h_w$ by 
\begin{equation}\label{eqn:truncated vertical distance}
h_w(y): = 
\begin{cases}
|y_d^{CS_{w}} - \psi_w(\wt y^{CS_{w}})| 1_{\mc U(w, R_1/2 )}(y), & \quad w \in \Si_j, \\
|y - x_l|1_{\mc U(w, R_1/2 )}(y), & \quad w= x_l.
\end{cases}
\end{equation}
Here, for any $y \in \mb R^d$, $(\wt y^{CS_{w}},y_d^{CS_{w}})$ are the coordinates of $y$ in $CS_w$.
If $w \in \Si_j$, we have
\begin{equation} 
B(w,R_1/2) \cap D = B(w, R_1/2)  \cap \{y= (\widetilde{y},y_d) \in CS_w: y_d > \psi_w(\widetilde{y})\}.
\end{equation} 
The following result 
extends Proposition \ref{key prop: perturbation} to general $C^{1, \beta}$ open sets.

\begin{prop} \label{perturbation: general case}
Suppose $w \in \partial D_\kappa$ is fixed. For any $p \in [\alpha - 1, \alpha) \cap (0, \alpha)$ when $w\in \partial D$ and  any $p \in (0, \alpha)$ when $w\in \cup_{l=1}^{m_1}\{x_l\}$, there exist constants $R \in \big(0, (R_1 \wedge 1)/2\big)$ and $\ga>0$, 
depending only on the characteristics of $D$,
and positive constant $A_2$, depending only on $p$ and the characteristic of $D$, such that for all $x \in B(w, R) \cap D_\kappa$, it holds that:
\begin{equation}
\begin{aligned}
& \bigg| \mathcal{A}(d, -\alpha) \, {\rm p.v.} 
\int_{D}
\frac{h_w(y)^{p} - h_w(x)^{p}}{|y-x|^{d+\alpha}} \, dy 
- C_k(d, \alpha, p) \, h_w(x)^p \, \delta_{D_\kappa}(x)^{- \alpha} \bigg| \\
& \leq A_2 \big(1 + \delta_{D_\kappa}(x)^{p-\alpha+\gamma} + |\log(\delta_{D_\kappa}(x))| \big), 
\quad k=0 \ 
\text{if} \ w \in \partial D; \quad k=1 \ \text{if} \ w \in \cup_{l=1}^{m_1}\{x_l\}.
\end{aligned}
\end{equation}
\end{prop}
\begin{proof}
We define the untruncated vertical distance $\bar{h}_w$ by :
\begin{equation}
\overline{h}_w(y): = 
\begin{cases}
|y_d^{CS_{w}} - \psi_w(\wt y^{CS_{w}})|, & \quad w \in \Si_j, j=1, \dots, m_0,\\
|y - x_l|, & \quad w= x_l, l=1, \dots, m_1.
\end{cases}
\end{equation}
\textbf{Case 1}: $w \in \Si_j, j=1, \dots, m_0$. 
By Proposition \ref{key prop: perturbation}, for any $p \in [\alpha - 1, \alpha) \cap (0, \alpha)$, there exist positive constants 
$r_1$ and $\ga$, depending only on the characteristics of $D$, and positive constant $A_1$, depending only on $p$ and the characteristics of $D$,
such that for any $x \in B(w,r_1\wedge R_1) \cap D$ (note that in this case $\delta_D(x) = \delta_{D_\kappa}(x)$),
\begin{align*}
& \bigg|\mathcal{A}(d,-\alpha){\ \rm p.v.}\int_{D_{\psi_j^w}} \frac{\overline{h}_w(y)^{p}-\overline{h}_w(x)^{p}}{|y-x|^{d+\alpha}}dy - C_0(d,\alpha, p) \overline{h}_w(x)^p \delta_{D}(x)^{- \alpha}\bigg| \\
&\le A_1(1+ \delta_{D}(x)^{p-\alpha+\gamma} + |\log(\delta_{D}(x))|).
\end{align*}
Note that if  $x \in B(w,R_1/2)\cap D$, then $\overline{h}_w(x) = h_w(x)$. 
Define $$R:= \frac{R_1\wedge r_1}{4}.$$
Then for any $x \in B(w,R)\cap D$,
\begin{align*}
& \bigg|\mathcal{A}(d,-\alpha){\ \rm p.v.}\int_{D} \frac{h_w(y)^{p}-h_w(x)^{p}}{|y-x|^{d+\alpha}}dy - C_0(d,\alpha,p) h_w(x)^p \delta_{D}(x)^{- \alpha}\bigg| \\
& \le \bigg|\mathcal{A}(d,-\alpha){\ \rm p.v.}\int_{D_{\psi_j^w}} \frac{\overline h_w(y)^{p}-\overline h_w(x)^{p}}{|y-x|^{d+\alpha}}dy - C_0(d,\alpha,p) \overline h_w(x)^p \delta_{D}(x)^{- \alpha}\bigg| + \mathcal{A}(d,-\alpha) \times \\
& \bigg( \int_{B(w,\frac{R_0 \wedge r_1}{2})^c \cap D_{\psi_j^w}} \frac{|\overline h_w(y)^{p}-\overline h_w(x)^{p}|}{|y-x|^{d+\alpha}}dy + \int_{B(w,\frac{R_0 \wedge r_1}{2})^c\cap D} \frac{|\overline h_w(y)^{p}-\overline h_w(x)^{p}|}{|y-x|^{d+\alpha}}dy \bigg) \\
& \lesssim 1+ \delta_D(x)^{p-\alpha +\gamma} + |\log (\delta_D(x))|.
\end{align*}
Here, the last inequality follows from Proposition \ref{key prop: perturbation} and similar argument as \eqref{non-singular argument}. 

\noindent \textbf{Case 2}: $w = x_l$ for some $l = 1,\cdots,m_1$. This case follows from the same argument except that we use $\mb R^d \backslash \{x_l\}$ instead of $D_{\psi_j^w}$ and Proposition \ref{harmonic: higher codimension}.
\end{proof}

\section{Proof of the main Theorem}
By Theorem \ref{t:vthke}, it suffices to show that for any $T>0$,
\begin{equation}\label{eqn:main target}
 \mb{P}^{x}(\tau_{D_\kappa}>t) \asymp \prod_{j=1}^{m_0} (1\wedge \frac{\delta_{\Si_j}(x)}{t^{1/\alpha}})^{p_j} \prod_{l=1}^{m_1} (1\wedge \frac{|x-x_l|}{t^{1/\alpha}})^{q_l}
\quad \forall (t, x)\in (0, T)\times D_\kappa,  
\end{equation}
where $p_{j}$ is determined  by $\lambda_{j}$ via $C_0(d,\alpha,p_{j}) = \lambda_{j}$ and $q_l$ is determined  by $\mu_{l}$ via $C_1(d,\alpha,q_{l}) = \mu_{l}$.

From now on, we fix $(t, x)\in (0, T)\times D_\kappa$. 
If $\delta_{D_\kappa}(x) \gtrsim t^{1/\alpha}$, then Lemma \ref{lemma:interior}, we have $\mb{P}^{x}(\tau_{D_\kappa}>t) \asymp 1$. 
Thus we can assume $\delta_{D_\kappa}(x) \lesssim t^{1/\alpha}$.

Let $w = w_x \in \partial D_\kappa$ be such that $\delta_{D_\kappa}(x) = |x - w|$. If $w\in \partial D$, then
\begin{equation} 
B(w,R_1/2) \cap D = B(w, R_1/2)  \cap \{y= (\widetilde{y},y_d) \in CS_w: y_d > \psi_w(\widetilde{y})\},
\end{equation} 
where $\psi_w$ is the $C^{1, \beta}$ function associated with $w$ in Definition \ref{d:c1beta}. 
Recall that by Proposition \ref{perturbation: general case},
for any $p \in [\alpha - 1, \alpha) \cap (0, \alpha)$, there exist constants $R \in \big(0, (R_1 \wedge 1)/2\big)$ and $\ga>0$, 
depending only on the characteristics of $D$, and positive constant $A_2$, depending only on $p$ and the characteristic of $D$, such that for any $z \in \mc U(w,R)$ (note that $\delta_D(z)=\delta_{D_\kappa}(z)$),
\begin{align*}
& \bigg|\mathcal{A}(d,-\alpha){\ \rm p.v.}\int_{D} \frac{h_w(y)^{p}-h_w(z)^{p}}{|y-z|^{d+\alpha}}dy - 
C_0(d,\alpha,p) h_w(z)^p \delta_{D}(z)^{- \alpha}\bigg| \\
&\le A_2(1+ \delta_{D}(z)^{p-\alpha+\gamma} + |\log(\delta_{D}(z))|), 
\end{align*}
where $h_w$ is defined by \eqref{eqn:truncated vertical distance}.

Now we construct appropriate super(sub)-harmonic functions. 
If $w \in \Si_j$ for some $j$, then we 
set $\wt{p_j} =\frac12\left(p_j+((p_j+\gamma)\wedge \alpha )\right)$ and define
\begin{equation}\label{test function: codim 1}
v_1(y):= h_w^{p_j}(y)+h_w^{\wt{p_j}}(y), \quad v_2(y):= h_w^{p_j}(y) - h_w^{\wt{p_j}}(y).
\end{equation}
Then the following lemma says that $v_1$ is a super-harmonic function and $v_2$ is a sub-harmonic function:

\begin{lemma} \label{test function estimate}
There exists $R_2 \in (0,R)$, depending only on the characteristics 
of $D$ and $p_j$ 
such that for any $z \in \mc U(w,R_2) = B(w,R_2)\cap D$,
\begin{equation}
\begin{aligned}
&(L^D_\alpha-\kappa)
v_1(z) \gtrsim \delta_D(z)^{\wt{p_j} - \alpha}, \\
&(L^D_\alpha-\kappa)
v_2(z) \lesssim -\delta_D(z)^{\wt{p_j} - \alpha},
\end{aligned}
\end{equation}
with implicit constants depend only on the characteristics of $D$.
\end{lemma}
\begin{proof}
Recalling $\lambda_{j} = C_0(d,\alpha,p_j)$ and $\delta_D(z) = \delta_{\Si_j}(z)$ for 
$z\in \mathcal{U}(w, R)$, we have
\begin{align*}
&(L^D_\alpha-\kappa)
v_1(z)\\
& = \mathcal{A}(d,-\alpha){\ \rm p.v.}\int_{D} \frac{h_w(y)^{p_j} + h_w(y)^{\wt p_j}- h_w(z)^{p_j} - h_w(z)^{\wt p_j}}{|y-z|^{d+\alpha}}dy - ( h_w(z)^{p_j} + h_w(z)^{\wt p_j}) \kappa(z) \\
& =\mathcal{A}(d,-\alpha){\ \rm p.v.}\int_{D} \frac{h_w(y)^{p_j}-h_w(z)^{p_j}}{|y-z|^{d+\alpha}}dy - C_0(d,\alpha,p_j) h_w(z)^{p_j} \delta_{D}(z)^{- \alpha} \\
& + \mathcal{A}(d,-\alpha){\ \rm p.v.}\int_{D} \frac{h_w(y)^{\wt p_j}-h_w(z)^{\wt p_j}}{|y-z|^{d+\alpha}}dy - C_0(d,\alpha,\wt p_j) h_w(z)^{\wt p_j} \delta_{D}(z)^{- \alpha} \\
& + (C_0(d,\alpha,\wt p_j) - C_0(d,\alpha,p_j)) h_w(z)^{\wt p_j} \delta_{D}(z)^{- \alpha} \\
& - \sum_{j' \neq j} \lambda_{j'} \delta_{\Si_{j'}}(z)^{-\alpha} v_1(z) - \sum_{l=1}^{m_1} |z - x_l|^{-\alpha} v_1(z)\\
& \ge -2A_2(1+ \delta_D(z)^{p_j-\alpha + \gamma} + |\log(\delta_D(z))| + \delta_D(z)^{\wt p_j-\alpha + \gamma})+ c_1 \delta_D(z)^{\wt p_j-\alpha} - c_2.
\end{align*}
For the inequality above, we used 
(i) Proposition \ref{perturbation: general case}; (ii) the strict increasing property of $C_0(d,\alpha,\cdot)$; and (iii) the facts that
$\delta_{\Si_{j}}(z) \ge R$ and $h_w(z) \asymp \delta_D(z)$ for $z \in \mc U(w,R)$. 
Since $\wt p_j \in (p_j, (p_j+\gamma)\wedge \alpha)$, $\delta_D(z)^{\wt p_j-\alpha}$ is the dominant term in the last line of the display above as $\delta_D(z)$ goes to zero. 
Note that the constants $c_1, c_2$  depend only on $p_j$ and the characteristics of $D$. 
Thus we can choose a constant $R_2 \in (0,R)$ depending only on 
$p_j$ and the characteristics of $D$ such that
$$
(L^D_\alpha-\kappa)
v_1(z) \gtrsim \delta_D(z)^{\wt p_j-\alpha}, z \in \mc U(w, R_2).
$$
Using the same arguments, we can show that 
$$
(L^D_\alpha-\kappa)
v_2(z) \lesssim -\delta_D(z)^{\wt p_j-\alpha}, z \in \mc U(w, R_2).$$
\end{proof}
Note that $v_1$ and $v_2$ do not belong to the domain 
of the generator $Y$, 
so Dynkin's formula cannot be applied directly. However, by employing a standard mollification procedure, see the discussion after \cite[Lemma 3.1]{CKSV20}, we can derive the following result, whose proof is provided in Appendix \ref{app:proof of mollify}.
\begin{lemma}\label{mollify: codim 1}
There exist $R_3 \in (0,R_2/2)$ depending only on $p_j$ and the characteristics of $D$, uniformly bounded $\{f_n\}, \{g_n\}\subseteq 
C_c^{\infty}(D_\kappa)$
and a strictly decreasing positive sequence $a_n \to 0$, such that 
\begin{itemize}
\item $\lim_{n\to \infty}f_n = v_1,\ \lim_{n\to \infty}g_n = v_2$ pointwisely on 
$\mc U(w,R_3)$.
\item For sufficiently large $n$, 
we have $(L^D_\alpha-\kappa)f_n(z)\ge 0$ and $(L^D_\alpha-\kappa)g_n(z)\le 0$ for any $z \in \mc U(w,R_3)$ with $\delta_D(z)>a_n$.
\end{itemize}
\end{lemma}
Using this lemma, we can get the following

\begin{lemma}\label{lemma:sub-super-harmonic}
    For any $z \in \mc U(w, R_3)$ and $r \in (0, R_3)$,
    \begin{equation}\label{ineqn:sub-super-harmonic}
       v_1(z) \le \mb E^z\big(v_1(Y_{\tau_{\mc U(z,r)}})\big),\quad v_2(z) \ge \mb E^z\big(v_1(Y_{\tau_{\mc U(z,r)}})\big).
    \end{equation}
\end{lemma}
\begin{proof}
 For any constant $s>0$, we define 
\begin{equation} \label{interior set}
D^{int}(s):= \{y\in D_\kappa: \delta_{D_\kappa}(y)>s\}
\end{equation}
Let $a_n$ be the sequence given in Lemma \ref{mollify: codim 1}.
Then we have 
$$
D^{int}(a_1) \subset\joinrel\subset \cdots \subset\joinrel\subset D^{int}(a_n) 
\to D_\kappa, 
$$
therefore
\begin{equation}
\lim_{n \to \infty}\tau_{\mc U(z,r) \cap D^{int}(a_n)} = \tau_{\mc U(z,r)}, \mb P - a.s.
\end{equation}
By Dynkin's formula, the bounded convergence theorem and the quasi-left continuity of $Y_t$, 
we have 
\begin{align*}
v_1(z) = \lim_{n\to \infty}f_n(z) & = \lim_{n\to \infty} \bigg[ \mb E^z\big(f_n(Y_{\tau_{\mc U(z,r) \cap D^{int}(a_n)}})\big) - \mb E^z\bigg(\int_0^{\tau_{\mc U(z,r) \cap D^{int}(a_n)}} \mc L^{\kappa}(f_n)(Y_s)ds \bigg) \bigg]\\
& \le \lim_{n\to \infty}\mb E^z\big(f_n(Y_{\tau_{\mc U(z,r) \cap D^{int}(a_n)}})\big) = \mb E^z\big(v_1(Y_{\tau_{\mc U(z,r)}})\big).
\end{align*}
The same argument holds for $v_2$. 
\end{proof}

We are now ready  to prove the main theorem.

\begin{proof}[Proof of Theorem \ref{thm:main}] 
We first assume $\delta_{D_\kappa}(x) = \delta_{\Si_j}(x) = |x-w|$ with $w \in \Si_j$. Define
\begin{equation} \label{radius of t}
c_1 := \frac{R_3 T^{-\frac{1}{\alpha}}}{4},
\quad r = r_t = c_1 t^{\frac{1}{\alpha}}.
\end{equation}
If $\delta_{D_{\kappa}}(x) \ge r$, then by Lemma \ref{lemma:interior}, we have 
$\mb P^x(\zeta >t) \asymp 1$. 
Thus we can assume $\delta_{D_{\kappa}}(x) = |x-w| < r = c_1t^{1/\alpha} < 
R_3/4$. Note that 
\begin{align}\label{hierarchy of inclusion}
 \mc U(x,r)\subseteq \mc U(w,2r) \subseteq \mc U(w, R_3) \subseteq \mc U(w,R).
\end{align}
Since $x \in \mc U(w, R_3)$, by Lemma \ref{lemma:sub-super-harmonic}, we have
\begin{align*}
&  v_1(x) \le \mb E^x\big(v_1(Y_{\tau_{\mc U(x,r)}})\big) \\
& = \mb E^x\big(v_1(Y_{\tau_{\mc U(x,r)}}); Y_{\tau_{\mc U(x,r)}} \in \mc U(x,2r)\big) + \sum_{n\ge 1} \mb E^x\big(v_1(Y_{\tau_{\mc U(x,r)}}); Y_{\tau_{\mc U(x,r)}} \in \mc U(x,2^{n+1}r)\backslash \mc U(x,2^{n}r)\big).
\end{align*}
For the first term, by the L\'evy system formula, we have 
\begin{equation}\label{lower prob estimate}
\begin{aligned}
\mb P^x(Y_{\tau_{\mc U(x,r)}} \in \mc U(x,2r)) \gtrsim \mb E^x \int_0^{\tau_{\mc U(x,r)}} \int_{\mc U(x,2r)} \frac{1}{|Y_s - z|^{d+\alpha}}dzds \gtrsim r^{-\alpha} \mb E^x(\tau_{\mc U(x,r)}).
\end{aligned}
\end{equation}
Moreover, for each $n\ge 1$, we have 
\begin{align*}
v_1(y) \lesssim h_w(y)^{p_j} \lesssim 2^{p_j n}r^{p_j},\ y \in \mc U(x,2^{n+1}r)\backslash \mc U(x,2^{n}r),
\end{align*}
where the first inequality follows from Lemma \ref{comparison: codimension 1}, and the second inequality follows from 
$h_w(y) \lesssim |y-x| + h_w(x) \lesssim 2^n r$. 
Thus 
\begin{align*}
& \mb E^x\big(v_1(Y_{\tau_{\mc U(x,r)}}); Y_{\tau_{\mc U(x,r)}} \in \mc U(x,2^{n+1}r)\backslash \mc U(x,2^{n}r)\big) \\
& \lesssim 2^{np_j}r^{p_j} \mb P^x(Y_{\tau_{\mc U(x,r)}} \in \mc U(x,2^{n+1}r)\backslash \mc U(x,2^{n}r)) \\
& \lesssim 2^{n p_j}r^{p_j} \mb E^x \int_0^{\tau_{\mc U(x,r)}} \int_{\mc U(x,2^{n+1}r)\backslash U(x,2^{n}r))}\frac{1}{|Y_s - z|^{d+\alpha}}dzds \\
& \lesssim 2^{n( p_j - \alpha)}r^{p_j} r^{-\alpha} \mb E^x(\tau_{\mc U(x,r)}) \\
& \lesssim 2^{n( p_j - \alpha)}r^{p_j} \mb P^x(Y_{\tau_{\mc U(x,r)}} \in \mc U(x,2r)),
\end{align*}
where in the last inequality, we used \eqref{lower prob estimate}. Therefore,
\begin{align*}
& \mb E^x\big(v_1(Y_{\tau_{\mc U(x,r)}})\big) \\
& = \mb E^x\big(v_1(Y_{\tau_{\mc U(x,r)}}); Y_{\tau_{\mc U(x,r)}} \in \mc U(x,2r)\big) + \sum_{n\ge 1} \mb E^x\big(v_1(Y_{\tau_{\mc U(x,r)}}); Y_{\tau_{\mc U(x,r)}} \in \mc U(x,2^{n+1}r)\backslash \mc U(x,2^{n}r)\big) \\
& \lesssim r^{p_j} \mb P^x(Y_{\tau_{\mc U(x,r)}} \in \mc U(x,2r)) + \sum_{n \ge 1} 2^{n(\wt p_j - \alpha)}r^{p_j} \mb P^x(Y_{\tau_{\mc U(x,r)}} \in \mc U(x,2r)) \\
& \lesssim r^{p_j} \mb P^x(Y_{\tau_{\mc U(x,r)}} \in \mc U(x,2r)) \le r^{p_j} 
\mb P^x(Y_{\tau_{\mc U(x,r)}} \in D_\kappa).
\end{align*}
Therefore by Lemma \ref{lifetime comparison}, we have 
\begin{align*}
\mb P^x(\zeta > t) 
\asymp \mb P^x(Y_{\tau_{\mc U(x,r)}} \in D_\kappa)
\gtrsim \frac{v_1(x)}{r^{p_j}} \gtrsim \frac{\delta_{D_{\kappa}}(x)^{p_j}}{t^{1/\alpha}},
\end{align*}
which gives the desired lower bound.

For the upper bound, we use
\begin{equation}
v_2(x) \ge \mb E^x(v_2(Y_{\tau_{\mc U(x,r)}})).
\end{equation}

\begin{center}
\begin{figure}[H]
\begin{tikzpicture}
\draw [->] (0,0) -- (8,0) node [right] {$\wt y$};
\node at (-.5,0) {$w$};
\node at (0,0.7) {$x$};
\fill (0,0.5) circle (.3mm);
\fill (0,0) circle (.3mm);
\node at (4,3) {$\mc U(w,R_3)$};
\node at (1.8,1.8) {$\mc U(x,r)$};
\node at (5.5,2.6) {$\psi_w$};
\node at (-2, 2.3) {$B_r$};
\draw[scale=0.8, domain=-8:8, smooth, variable=\x, blue] plot ({\x}, {abs(\x)*sqrt(abs(\x))/5});
\draw[scale=0.8, domain=-1.84:1.84, smooth, variable=\x, black] plot ({\x}, {0.5 + sqrt(1.84*1.84 - \x*\x)});
\draw[scale=0.8, domain=-4.59:4.59, smooth, variable=\x, black] plot ({\x}, {sqrt(25 - \x*\x)});
\draw[thick, black] (-2, 2.3) circle (0.5);
\end{tikzpicture}
\caption{Illustration of the choice of $B_r$.}
\label{figure:ball}
\end{figure}
\end{center}
By  \eqref{hierarchy of inclusion}, 
we can choose a ball $B_r\subseteq \mc U(w,R)$ centered at some point inside $\mc U(w,R)$ with radius comparable to $r \asymp t^{1/\alpha}$, see Figure \ref{figure:ball} for an illustration, such that 
\begin{align*}
& B_r \cap \mc U(x,r) = \emptyset,  \quad \dist(B_r, \mc U(x,r)) \gtrsim r, \\
& \delta_{D_{\kappa}}(y) \asymp r,\ \forall y \in B_r.
\end{align*} 
Hence
\begin{align*}
\delta_{D_{\kappa}}(x)^{p_j} & \gtrsim v_2(x) \ge \mb E^x(v_2(Y_{\tau_{\mc U(x,r)}})) \ge E^x(v_2(Y_{\tau_{\mc U(x,r)}}); Y_{\tau_{\mc U(x,r)}} \in B_r) \\
& \gtrsim (r^{p_j} - r^{\wt p_j}) \mb P^{x}(Y_{\tau_{\mc U(x,r)}} \in B_r) \gtrsim r^{p_j}\mb P^{x}(Y_{\tau_{\mc U(x,r)}} \in B_r), 
\end{align*}
where in the last inequality we used $p_j<\wt p_j$. Thus
by Lemma \ref{lifetime comparison}, we have 
\begin{align*}
\mb P^x(\zeta > t) \asymp \mb P^x(Y_{\tau_{\mc U(x,r)}} \in D_\kappa)
\asymp \mb P^{x}(Y_{\tau_{\mc U(x,r)}} \in B_r) \lesssim \frac{\delta_{D_{\kappa}}(x)^{p_j}}{t^{1/\alpha}}.
\end{align*}
By \eqref{e:R1}, we have 
\begin{align*}
\delta_{\Si_{j'}}(x) \ge R_1/2, \quad |x-x_l|\ge R_1/2, \quad j' \neq j, l=1, \dots, m_1.
\end{align*}
Therefore, we have 
\begin{align*}
1 \wedge \frac{\delta_{\Si_{j'}}(x)}{t^{1/\alpha}} \asymp 1, \quad 1 \wedge \frac{|x-x_l|}{t^{1/\alpha}} \asymp 1,
\quad j'\neq j; \ 1\le l\le m_1.
\end{align*}
In conclusion, we have shown that 
\begin{align*}
\mb P^x(\zeta >t) 
\asymp \prod_{j=1}^{m_0} (1\wedge \frac{\delta_{\Si_j}(x)}{t^{1/\alpha}})^{p_j} \prod_{l=1}^{m_1} (1\wedge \frac{|x-x_l|}{t^{1/\alpha}})^{q_l}.
\end{align*}

Finally we discuss the case when $w = x_l$ for some $l$. In this case, we choose the sub- and super-harmonic functions given by 
\begin{equation}
\begin{aligned}
& \hat{v}_1(y) := h_w(y)^{q_l}+ h_w(y)^{\wt q_l},\quad \hat{v}_2(y):= h_w(y)^{q_l} - h_w(y)^{\wt q_l},\quad \wt q_l =\frac12(q_l+\alpha).
\end{aligned}
\end{equation}
Applying the same procedure as the case of $C^{1,\beta}$ open set, we arrive at the same conclusion that 
\begin{align*}
\mb P^x(\zeta >t) 
\asymp \prod_{j=1}^{m_0} (1\wedge \frac{\delta_{\Si_j}(x)}{t^{1/\alpha}})^{p_j} \prod_{l=1}^{m_1} (1\wedge \frac{|x-x_l|}{t^{1/\alpha}})^{q_l}.
\end{align*}
This concludes the proof of Theorem \ref{thm:main}.   
\end{proof}

\appendix
\section{Proof of Lemma \ref{lifetime comparison}} \label{app:proof of lifetime}
\begin{proof} 
First, note that 
\[
\{\zeta > t\} \subseteq \{\tau_{\mc U(x,at^{1/\alpha})}>t\} \cup 
\{Y_{\tau_{\mc U(x,at^{1/\alpha})}} \in D_\kappa\}.
\]
This gives the estimate
$$
\mb{P}^x(\zeta > t) \le \mb{P}^x(\tau_{\mc U(x,at^{1/\alpha})}>t) + 
\mb{P}^x(Y_{\tau_{\mc U(x,at^{1/\alpha})}} \in D_\kappa) 
\le \frac{\mb E^x[\tau_{\mc U(x,at^{1/\alpha})}]}{t} + 
\mb{P}^x(Y_{\tau_{\mc U(x,at^{1/\alpha})}} \in D_\kappa).
$$
We now construct a suitable ball $W$ inside 
$D_\kappa$ that is sufficiently separated from $\mc U(x,at^{1/\alpha})$ 
but not too separated. Let $y\in D_\kappa$ and choose 
$W = B(y,b t^{1/\alpha}) \subseteq D_\kappa$ 
for some $b<a/8$, such that 
\begin{equation}\label{choose W: far}
B(y,2b t^{1/\alpha}) \subseteq \mc A(x,(a+6b)t^{1/\alpha}, (a+2b)t^{1/\alpha})\cap D_\kappa.
\end{equation}
With this choice, $W$ satisfies the four properties required by the lemma.

By condition {\bf (U)}, see \eqref{key upper bound}, we can choose a function 
$f \in \mc D\big(\overline{\mc A(x, (a+b)t^{1/\alpha}, a t^{1/\alpha})}\cap \overline{D}, \mc A(x,(a+2b)t^{1/\alpha}, (a-b)t^{1/\alpha})\cap\overline{D}\big)$ 
such that 
\begin{equation}
\sup_{z \in \overline{D}} 
\mc L^{\overline{D}}_\alpha f(z) \lesssim \frac{1}{t}.
\end{equation}
By the definition of $Y$ and Dynkin's formula (see \cite[(2.11)]{Bogdan15} and the proof of \cite[(4.6)]{Bogdan15}), we have
\begin{align*}
& \mb P^x\big(Y_{\tau_{\mc U(x,at^{1/\alpha})}} \in \mc U(x,(a+b)t^{1/\alpha})\big) \\
&=\mb E^x\left[\exp\left(-\int^{\tau_{\mc U(x,at^{1/\alpha})}}_0
\kappa(X_s)ds\right) ; 
X_{\tau_{\mc U(x,at^{1/\alpha})}} \in \mc U(x,(a+b)t^{1/\alpha})\right]\\
& \le \mb E^x\left[f(X_{\tau_{\mc U(x,at^{1/\alpha})}})\exp\left(-\int^{\tau_{\mc U(x,at^{1/\alpha})}}_0
\kappa(X_s)ds\right)\right]\\
&= \mb E^x\bigg(\int_0^{\tau_{\mc U(x,at^{1/\alpha})}}(
\mc L^{\overline{D}}_\alpha  f(X_s) 
- \kappa(X_s)f(X_s))ds\bigg) \\
& \le \sup_{z\in D} 
\mc L^{\overline{D}}_\alpha f(z)
\mb E^x[\tau_{\mc U(x,at^{1/\alpha})}] \lesssim \frac{\mb E^x(\tau_{\mc U(x,at^{1/\alpha})})}{t}.
\end{align*}
Using the L\'evy system formula, for the event 
$Y_{\tau_{\mc U(x,at^{1/\alpha})}} \in D_\kappa \setminus \mc U(x,(a+b)t^{1/\alpha})$, 
we have 
\begin{align*}
& \mb P^x(Y_{\tau_{\mc U(x,at^{1/\alpha})}} \in 
D_\kappa \backslash \mc U(x,(a+b)t^{1/\alpha})) \\
&=\mb E^x \left[\int_0^{\tau_{\mc U(x,at^{1/\alpha})}} 
\int_{D_\kappa \backslash \mc U(x,(a+b)t^{1/\alpha})}J(Y_s,w)m(dw)ds\right] \\
& \lesssim \mb E^x(\tau_{\mc U(x,at^{1/\alpha})})\int_{|w-x|\ge b t^{1/\alpha}}|w-x|^{-d-\alpha}m(dw) \lesssim \frac{\mb E^x(\tau_{\mc U(x,at^{1/\alpha})})}{t}.
\end{align*}
Similarly, by choosing the $W$ satisfying \eqref{choose W: far}, we have 
\begin{align*}
\mb P^x(Y_{\tau_{\mc U(x,at^{1/\alpha})}} \in W) \asymp \frac{\mb E^x(\tau_{\mc U(x,at^{1/\alpha})})}{t}.
\end{align*}
Combining these estimates, we see
\begin{align*}
& \mb{P}^x(\zeta > t) \le \frac{\mb E^x(\tau_{\mc U(x,at^{1/\alpha})})}{t} + 
\mb{P}^x(Y_{\tau_{\mc U(x,at^{1/\alpha})}} \in D_\kappa) \\
&= \frac{\mb E^x(\tau_{\mc U(x,at^{1/\alpha})})}{t} + \mb P^x(Y_{\tau_{\mc U(x,at^{1/\alpha})}} \in \mc U(x,(a+b)t^{1/\alpha})) 
+ \mb P^x(Y_{\tau_{\mc U(x,at^{1/\alpha})}} \in D_\kappa \backslash \mc U(x,(a+b)t^{1/\alpha})) \\
& \lesssim \frac{\mb E^x(\tau_{\mc U(x,at^{1/\alpha})})}{t} \asymp \mb P^x(Y_{\tau_{\mc U(x,at^{1/\alpha})}} \in W).
\end{align*}

For the lower bound, note that $\zeta \ge \tau_{\mc U(x,(a+6b)t^{1/\alpha})}$. Using the strong Markov property and the construction of $W$, we have:
\begin{align*}
& \mb P^x(\zeta > t) \ge \mb P^x(\tau_{\mc U(x,(a+6b)t^{1/\alpha})} > t) \\
& \ge \mb E^x\big(\mb P^{Y_{\tau_{\mc U(x,at^{1/\alpha})}}}(\tau_{\mc U(x,(a+6b)t^{1/\alpha})} > t); Y_{\tau_{\mc U(x,at^{1/\alpha})}} \in W\big) \\
& \ge \inf_{z \in W} \mb P^z(\tau_{B(y,2bt^{1/\alpha})}>t) \mb P^x(Y_{\tau_{\mc U(x,at^{1/\alpha})}} \in W) \gtrsim \mb P^x(Y_{\tau_{\mc U(x,at^{1/\alpha})}} \in W),
\end{align*}
where we used that $B(y,2bt^{1/\alpha}) \subseteq \mc U(x,(a+6b)t^{1/\alpha})$ (from \eqref{choose W: far}) and the interior estimate (Lemma \ref{lemma:interior}) to get the uniform lower bound.

Thus, we have shown both the upper and lower bounds hold, completing the proof.
\end{proof}

\section{Proof of Lemma \ref{mollify: codim 1}}\label{app:proof of mollify}

\begin{proof}[Proof of Lemma \ref{mollify: codim 1}]
In this proof, we assume $z\in \mc U(w, R_2/2)$.
     First we choose a nonnegative function $\varphi \in C_c^{\infty}(\mb R^d)$ such that $\text{supp}\ \varphi \subseteq B(0,1)$ and $\int_{\mb R^d} \varphi(u)du = 1$. 
Let $\{a_n\},\{b_n\},\{c_n\}$ be three sequences of positive numbers, all tending to zero, to be specified later.
Recall the interior set $D_{int}(s)$ is defined in \eqref{interior set}, we define
\begin{align*}
& \varphi_n(y): = c_n^{-d} \varphi(c_n^{-1} y),  f_n := \varphi_n *(v_1 1_{D^{int}(b_n)}),\ g_n := \varphi_n *(v_2 1_{D^{int}(b_n)}).
\end{align*}
We claim that for appropriately chosen $\{a_n\},\{b_n\},\{c_n\}$, there exists a constant 
$R_3 \in (0, R_2)$ depending only on the characteristics of $D$ such that for sufficiently large $n$,  
\begin{align*}
(L^D_\alpha-\kappa)f_n(z)\ge 0,\ (L^D_\alpha-\kappa)g_n(z)\le 0, 
\quad z \in \mc U(w, R_3) \cap D^{int}(a_n).
\end{align*}
Indeed, it holds that
\begin{align*}
&(L^D_\alpha-\kappa)f_n(z)=(L^D_\alpha-\kappa)(\varphi_n * v_1)(z)-(L^D_\alpha-\kappa)(\varphi_n * v_1 - f_n)(z) \\
&= (L^D_\alpha-\kappa)(\varphi_n * v_1)(z)+ \mc A(d,-\alpha) \times \\
&  {\ \rm p.v.}\int_{D} \int_{\mb R^d} \varphi_n(u) \frac{\big(1_{D^{int}(b_n)}(y-u)- 1 \big)v_1(y-u) - \big(1_{D^{int}(b_n)}(z-u) - 1 \big)v_1(z-u)}{|y-z|^{d+\alpha}} dudy \\
& + \kappa(z) \int_{\mb R^d} \varphi_n(u)\big(1_{D^{int}(b_n)}(z-u) - 1 \big)v_1(z-u)du\\
& =: I + II + III.
\end{align*}
Let us discuss $III, II$ first, which are simpler. Since $\text{supp}\ \varphi_n \subseteq B(0,c_n)$, if
\begin{equation}\label{sequence assump: 1}
\lim_{n\to \infty} \frac{c_n}{a_n} = 0,\quad \lim_{n\to \infty} \frac{b_n}{a_n} = 0,
\end{equation} 
for $n$ large enough we have $z - u \in D^{int}(b_n)$ when $z \in D^{int}(a_n)$ and $|u|<c_n$, and thus
$$
1_{D^{int}(b_n)}(z-u) - 1 = 0.
$$
Hence $III$ is equal to zero for $n$ large enough. 

For the estimate of $II$, we have 
\begin{align*}
|II|& = \mc A(d,-\alpha) {\ \rm p.v.}\int_{D} \int_{\mb R^d} \varphi_n(u) \frac{\big(1 - 1_{D^{int}(b_n)}(y-u)\big)v_1(y-u)}{|y-z|^{d+\alpha}} dudy \\
& \lesssim \int_{y \in \mc U(w,R), 
\delta_{D_\kappa}(y)\le 2b_n} 
\int_{\mb R^d} \varphi_n(u)\frac{\delta_{D}(y-u)^{p_j}}{|y-z|^{d+\alpha}} du dy \\
& \lesssim b_n^{p_j} \int_{y \in \mc U(w,R), 
\delta_{D_\kappa}(y)\le 2b_n} 
\frac{1}{|y-z|^{d+\alpha}}dy  \\
& \lesssim b_n^{p_j} (a_n - 2b_n)^{-d-\alpha}.
\end{align*}
In the first inequality above, we used the following facts:
(i) $1 - 1_{D^{int}(b_n)}(y-u) = 0$ 
when $\delta_{D_\kappa}(y)>2b_n$
(since $\delta_{D_\kappa}(y-u) \ge 2b_n  -c_n >b_n$); 
and (ii) $v_1(y) \lesssim \delta_{D_\kappa}(y)^{p_j}$ for any $y \in \mc U(w,R_2)$. 
In the second inequality of the display above, we used $\delta_{D_\kappa}(y-u) \le \delta_{D_\kappa}(y)+|u| \lesssim b_n + c_n \lesssim b_n$.
In the last inequality of the display above, we use the fact that,  for $\delta_{D_\kappa}(y)\le 2b_n$ and $\delta_{D_\kappa}(z)>a_n$, 
we have $|y-z| \ge a_n - 2b_n$, which is greater than $0$ for $n$ large under the assumption \eqref{sequence assump: 1}. 

In conclusion, we have shown that if we assume 
\begin{equation} \label{sequence assump: 2}
\lim_{n \to \infty} b_n^{p_j} (a_n - 2b_n)^{-d-\alpha} = 0,
\end{equation}
then $|II|$ tends to 0 as $n\to\infty$.

Now we turn to the estimate of $I$: 
\begin{align*}
I &= {\rm p.v.}\int_D \int_{\mb R^d} \varphi_n(u) \frac{v_1(y-u) - v_1(z-u)}{|y-z|^{d+\alpha}}du dy - \kappa(z) \int_{\mb R^d} \varphi_n(u)v_1(z-u)du \\
& = \int_{\mb R^d} \bigg({\rm p.v.}\int_D \frac{v_1(y-u) - v_1(z-u)}{|y-z|^{d+\alpha}}dy - \kappa(z-u)v_1(z-u) \bigg) \varphi_n(u)du \\
& +  \int_{\mb R^d} \varphi_n(u)(\kappa(z-u)-\kappa(z))v_1(z-u)du \\
& =\int_{\mb R^d} \bigg( {\rm p.v.}\int_{D-u} \frac{v_1(y) - v_1(z-u)}{|y-(z-u)|^{d+\alpha}}dy - \kappa(z-u)v_1(z-u) \bigg) \varphi_n(u)du \\
& +  \int_{\mb R^d} \varphi_n(u)(\kappa(z-u)-\kappa(z))v_1(z-u)du \\
& = \int_{\mb R^d} \bigg({\ \rm p.v.}\int_{D} \frac{v_1(y) - v_1(z-u)}{|y-(z-u)|^{d+\alpha}}dy - \kappa(z-u)v_1(z-u) \bigg) \varphi_n(u)du \\
& + \bigg({\rm p.v.}\int_{D -u} \frac{v_1(y) - v_1(z-u)}{|y-(z-u)|^{d+\alpha}}dy   - \text{p.v.}\int_{D} \frac{v_1(y) - v_1(z-u)}{|y-(z-u)|^{d+\alpha}}dy \bigg) \\
& +  \int_{\mb R^d} \varphi_n(u)(\kappa(z-u)-\kappa(z))v_1(z-u)du \\
& =: I_1 + I_2 + I_3.
\end{align*}
In the first equality above, we used Fubini's theorem.
To justify the application of Fubini's theorem, we use the standard estimate: for $|u|<c_n$,
\begin{align*}
\sup_{\varepsilon \in (0,1)} \left|\int_{D, |y-z|\ge \varepsilon} \frac{h_w(y-u)^p - h_w(z-u)^p}{|y-z|^{d+\alpha}}dy \right| \lesssim 
\delta_{D}(z-u)^{p-\alpha}.
\end{align*}
This standard estimate can be checked in the same way as in \cite[(3.14)]{CKSV20}. 
Since $v_1 = h_w^{p_j} + h_w^{\wt p_j}$, Fubini's theorem can be applied.

Now we fix $|u|<c_n$. For $z \in \mc U(w, R_2/2)$, when $n$ is large we have $z-u \in  \mc U(w, R_2)$. So we can apply Lemma \ref{test function estimate} to get
\begin{align*}
I_1 & = \int_{\mb R^d} \bigg({\rm p.v.}\int_{D} \frac{v_1(y) - v_1(z-u)}{|y-(z-u)|^{d+\alpha}}dy - \kappa(z-u)v_1(z-u) \bigg) \varphi_n(u)du \gtrsim \delta_{D_{\kappa}}(z-u)^{\wt p_j}.
\end{align*}
To estimate $I_2$,  we first note that if 
$y \in (D-u)\backslash D \subseteq D^c$, $v_1(y) = 0$ by definition. 
Moreover, for $y \in D\backslash (D-u)$, we have $v_1(y) \lesssim |u|^{p_j} \le c_n^{p_j}$. 
For $\delta_{D_{\kappa}}(z)>a_n$, we have $\delta_{D_{\kappa}}(z-u) \gtrsim a_n$, thus $v_1(z-u) \gtrsim a_n^{p_j}$. Therefore for large $n$, under the assumption \eqref{sequence assump: 1},  we have
$$\int_{D \backslash (D-u)} \frac{v_1(y) - v_1(z-u)}{|y-(z-u)|^{d+\alpha}}dy \le 0.$$
Hence we have 
\begin{align*}
I_2 & = {\rm p.v.}\int_{D -u} \frac{v_1(y) - v_1(z-u)}{|y-(z-u)|^{d+\alpha}}dy   - {\rm p.v.}\int_{D} \frac{v_1(y) - v_1(z-u)}{|y-(z-u)|^{d+\alpha}}dy \\
& = \int_{(D-u)\backslash D} \frac{v_1(y) - v_1(z-u)}{|y-(z-u)|^{d+\alpha}}dy - \int_{D \backslash (D-u)} \frac{v_1(y) - v_1(z-u)}{|y-(z-u)|^{d+\alpha}}dy \\
& \ge -v_1(z-u) \int_{(D-u)\backslash D} \frac{1}{|y-(z-u)|^{d+\alpha}}dy.
\end{align*} 
Next we try to get an upper bound on the integral above.
Choose $w_u \in \partial D$ (in fact, $w_u \in \Si_{j}$) with $\delta_{D_{\kappa}}(z-u) = |z-u - w_u|$ such that 
$$
B(w_u,R_0)\cap D = B(w_u,R_0) \cap \{y = (\wt y, y_d): y_d > 
\psi_{w_u}(\wt y)\},
$$
where $\psi_{w_u}$ is the $C^{1,\be}$ function associated with $w_u$ in Definition \ref{d:c1beta}. Then
\begin{align*}
& \int_{(D-u)\backslash D} \frac{1}{|y-(z-u)|^{d+\alpha}}dy \\
& = 
\int_{(D-u)\backslash D, B(w_u,R_2)} \frac{1}{|y-(z-u)|^{d+\alpha}}dy
+ \int_{(D-u)\backslash D, B(w_u,R_2)^c}
 \frac{1}{|y-(z-u)|^{d+\alpha}}dy \\
& \lesssim \int_{|\wt y|< \frac{\delta_{D_{\kappa}}(z-u)}{2\Lambda}} 
\int_{\psi_{w_u}(\wt y) - |u|}^{\psi_{w_u}(\wt y)} 
|y_d - \delta_{D_{\kappa}}(z-u)|^{-d-\alpha}dy_d d\wt y  \\
& + \int_{\frac{\delta_{D_{\kappa}}(z-u)}{2\Lambda} \le |\wt y|< R_2} 
\int_{\psi_{w_u}(\wt y) - |u|}^{\psi_{w_u}(\wt y)} 
|\wt y|^{-d-\alpha} dy_d d\wt y + \int_{B(z-u, R_2)^c} \frac{1}{|y-(z-u)|^{d+\alpha}}dy  \\
& \lesssim |u|\delta_{D_{\kappa}}(z-u)^{-\alpha - 1} + 1.
\end{align*}
For the last inequality, we used the fact that, when $|\wt y|< \frac{\delta_{D_{\kappa}}(z-u)}{2\Lambda}$,  we have $|\psi^u(\wt y)| \le \frac{1}{2\Lambda} \Lambda \delta_{D_{\kappa}}(z-u)$ and thus
$$|y_d - \delta_{D_{\kappa}}(z-u)|^{-d-\alpha} \lesssim \delta_{D_{\kappa}}(z-u)^{-d-\alpha}.$$
The remaining estimates follow from valuation of the integrals. In summary, we have
\begin{align}\label{e:I2}
I_2 \gtrsim -c_n \delta_{D_{\kappa}}(z-u)^{p_j-\alpha - 1} - \delta_{D_{\kappa}}(z-u)^{p_j} \gtrsim -\frac{c_n}{a_n^{-p_j+\alpha + 1}}- \delta_{D_{\kappa}}(z-u)^{p_j}.
\end{align}
Under the assumption 
\begin{equation}\label{sequence assump: 3}
\lim_{n\to \infty} \frac{c_n}{a_n^{-p_j+\alpha + 1}} = 0,
\end{equation}
we have $I_2\gtrsim - \delta_{D_{\kappa}}(z-u)^{p_j}$ when $n$ is large enough.

For $I_3$, we have
\begin{align*}
I_3& = \bigg|\int_{\mb R^d}\sum_{j}\lambda_{j}(\delta_{\Si_{j}}(z-u)^{-\alpha} -\delta_{\Si_{j}}(z)^{-\alpha} ) \varphi_n(u)v_1(z-u)du \bigg| \\
& \lesssim \int_{\mb R^d} \big|\delta_{\Si_{j}}(z-u)^{-\alpha} -\delta_{\Si_{j}}(z)^{-\alpha} \big|\varphi_n(u)v_1(z-u)du + \int_{\mb R^d}\varphi_n(u)v_1(z-u)du \\
& \lesssim \frac{c_n}{a_n^{-p_j+\alpha +1}}+ \delta_{D_{\kappa}}(z-u)^{p_j} ,
\end{align*}
where for the last inequality, we used the mean-value theorem for the function $f(x) = x^{-\alpha}$ and the fact the $|\delta_{\Si_{j}}(z-u)-\delta_{\Si_{j}}(z)| \lesssim |u|$. 
Under the assumption \eqref{sequence assump: 3}, the first term on the last line of the display above goes to 0 as $n\to\infty$.

In summary, under the assumptions that $\{a_n\},\{b_n\},\{c_n\}$ tend to zero and satisfy \eqref{sequence assump: 1},\eqref{sequence assump: 2} and \eqref{sequence assump: 3}, for 
$n$ large enough we have 
\begin{align*}
(L^D_\alpha-\kappa)f_n(z)
\gtrsim \delta_{D_{\kappa}}(z-u)^{\wt p_j-\alpha} - 1 - \delta_{D_{\kappa}}(z-u)^{p_j} \gtrsim \delta_{D_{\kappa}}(z)^{\wt p_j-\alpha} - 1 - \delta_{D_{\kappa}}(z)^{p_j}.
\end{align*}
Note that since $\delta_{D_{\kappa}}(z)<\wt R_2$ and $\wt p_j-\alpha<0$, we can choose 
$R_3\in (0, R_2)$, depending only on $p_j$ and the characteristics of $D$, 
such that 
for $z\in \mc U(w, R_3)$ and $n$ large enough 
$$
(L^D_\alpha-\kappa)f_n(z)
\gtrsim \delta_{D_{\kappa}}(z)^{\wt p_j-\alpha} \ge 0.$$
The same argument can be applied to $g_n$.  The proof is complete.
\end{proof}

\end{document}